\documentclass[11pt]{article}
\usepackage{amsthm, amsmath, amssymb, amsfonts, url, booktabs, tikz, setspace, fancyhdr, bm,xcolor}

\usepackage{mathrsfs} 
\usepackage[margin = 2.4cm]{geometry}
\usepackage{hyperref, enumerate}
\usepackage[shortlabels]{enumitem}
\usepackage[babel]{microtype}
\usepackage[english]{babel}
\usepackage[capitalise]{cleveref}
\usepackage{comment}
\usepackage{bbm}
\usepackage{csquotes}
\usepackage{mathabx}
\usepackage{tikz}
\usepackage{graphicx}
\usepackage{float}
\usepackage{thm-restate}
\usepackage{mathtools}
\usepackage{booktabs}
\usepackage{multirow}

\usepackage{tikz}
\usepackage{graphicx}
\usepackage{float}
\usepackage{subcaption} % 导言区加入
\usetikzlibrary{positioning}
\usetikzlibrary{shapes,arrows,arrows,positioning,fit}
\usetikzlibrary{decorations.pathreplacing}
\counterwithin{figure}{section}

\newtheorem{theorem}{Theorem}[section]
\newtheorem{prop}[theorem]{Proposition}
\newtheorem{lemma}[theorem]{Lemma}
\newtheorem{cor}[theorem]{Corollary}
\newtheorem{conj}[theorem]{Conjecture}
\newtheorem{cons}{Construction}
\newtheorem{claim}[theorem]{Claim}
\newtheorem{case}{Case}

\newtheorem{fact}[theorem]{Fact}

\newtheorem*{property}{Property}
  
\theoremstyle{definition}
\newtheorem{definition}[theorem]{Definition}

\definecolor{mypink}{RGB}{255, 65, 172}
\definecolor{myorange}{RGB}{255, 178, 49}
\definecolor{darktangerine}{rgb}{1.0, 0.66, 0.07}
\definecolor{darkpastelgreen}{rgb}{0.01, 0.75, 0.24}
\theoremstyle{definition}
\newtheorem{defn}[theorem]{Definition}
\newtheorem*{defn-non}{Definition}

\usepackage{todonotes}

\newlist{Case}{enumerate}{2}
\setlist[Case, 1]{%
    label           =   {\bfseries Case \arabic*.},
    labelindent=1em ,labelwidth=1.3cm, labelsep*=1em, leftmargin =!
}
\setlist[Case, 2]{%
    label           =   {\bfseries Subcase \arabic{Casei}.\arabic*.},
    labelindent=-1em ,labelwidth=1.3cm, labelsep*=1em, leftmargin =!
}

\newenvironment{poc}{\begin{proof}[Proof of claim]}{\end{proof}}

\newcommand{\C}[1]{{\protect\mathcal{#1}}}

\newcommand{\ceil}[1]{\lceil #1\rceil}

\usepackage{todonotes} 

\newcommand{\eps}{\varepsilon}

\title{Chromatic thresholds for pairs of graphs}
\author{Jun Gao \thanks{Mathematics Institute, University of Warwick, Coventry, UK. Supported by ERC Advanced Grant 101020255.}
\and Hong Liu \thanks{Extremal Combinatorics and Probability Group (ECOPRO), Institute for Basic Science (IBS),  Daejeon, South Korea. Supported by the Institute for Basic Science (IBS-R029-C4).}
\and Zhuo Wu\thanks{Departament de Matemàtiques, Universitat Politècnica de Catalunya (UPC),
Carrer de Pau Gargallo 14, 08028 Barcelona, Spain. Z. Wu acknowledges the bilateral AEI+DFG research project PCI2024-155080-2: SRC-ExCo – Structure, Randomness and Computational Methods in Extremal Combinatorics, and the PID2023-147202NB-I00 (COCOA: COntemporary COmbinatorics and its Applications), all funded by MICIU/AEI/10.13039/501100011033. Email: \texttt{zhuo.wu@upc.edu}.}
\and Yisai Xue\thanks{School of Mathematics and Statistics, Ningbo University, Ningbo, China. Supported by the National Nature Science Foundation of China (No. 12501486). Email: xueyisai@nbu.edu.cn}}

\begin{document}

\maketitle

\begin{abstract}
The chromatic threshold of a graph $H$ is the minimum-degree density above which every $H$-free graph has bounded chromatic number. We study a two-color Ramsey analogue: for graphs $H_1$ and $H_2$, we ask for the minimum-degree density above which every graph that admits a red-blue edge-coloring with no red copy of $H_1$ and no blue copy of $H_2$ has bounded chromatic number. We give a complete answer when both $H_1$ and $H_2$ are 3-chromatic. The threshold takes exactly one of the five values
\[
\frac23,\quad \frac57,\quad \frac34,\quad \frac79,\quad \frac45,
\]
and we characterize precisely which pairs $(H_1,H_2)$ give each value. The classification is determined by the ordinary chromatic thresholds of $H_1$ and $H_2$ and by their embeddability into a hierarchy of $C_5$-type Ramsey configurations.
\end{abstract}

\section{Introduction}
In 1973, Erd\H{o}s and Simonovits~\cite{erdos1973valence} initiated a systematic study of the minimum-degree condition above which forbidding a fixed graph $H$ forces bounded chromatic number. They introduced  the \emph{chromatic threshold} of a graph $H$, defined as
\begin{align*}
    \delta_\chi(H):=
    \inf \{d: \exists~C&=C(H, d)~\text{ s.t. $\forall$ $n$-vertex $H$-free graph $G$}, \\
    & \text{ if $\delta(G) \geq d n$, then $\chi(G) \leq C$}\}.
\end{align*}
Thus, for $d<\delta_{\chi}(H)$ the chromatic number of $H$-free graphs with minimum degree $dn$ may be arbitrarily large, while for $d>\delta_{\chi}(H)$ it is necessarily bounded.

Erd\H{o}s and Simonovits conjectured that $\delta_\chi(K_3)=1/3$, which was proved by Thomassen~\cite{thomassen2002chromatic}. Subsequently, Goddard and Lyle~\cite{goddard2011dense} and Nikiforov~\cite{nikiforov2010chromatic} extended this to complete graphs, showing that $\delta_\chi(K_r)=\frac{2r-5}{2r-3}.$
Additionally, \L uczak and Thomass\'e~\cite{luczak2010coloring} identified a large family of graphs with chromatic threshold zero. Finally, Allen, B{\"o}ttcher, Griffiths, Kohayakawa and Morris~\cite{allen2013chromatic} determined $\delta_\chi(H)$ for all graphs $H$, showing that $\delta_\chi(H) \in\left\{\frac{r-3}{r-2},~\frac{2 r-5}{2 r-3},~\frac{r-2}{r-1}\right\}$ where $\chi(H)=r \geq 3$.

%In particular, when $\chi(H)=3$, the possible values are $0$, $1/3$ and $1/2$. These three values already reflect substantial structural information about $H$: for example, near-acyclic graphs have threshold zero, while graphs whose decomposition family contains no forest have threshold $1/2$.

Recent work has developed chromatic thresholds in several directions, including stability near extremal thresholds~\cite{liu-shangguan-wu-xue}, connections with the $(p,q)$-theorem in discrete geometry and blow-up phenomena~\cite{liu2024beyond}, interaction with VC-dimension theory and homomorphism threshold~\cite{huang2025interpolating}, and analogues for linear equations over finite fields~\cite{liu2026chromatic} and applications in topological dynamics.

Another fundamental topic in graph theory is Ramsey theory. Given graphs $G$, $H_1$ and $H_2$, we say that $G$ is \emph{$(H_1,H_2)$-Ramsey} if every $2$-edge-coloring of $G$ contains a monochromatic copy of $H_i$ in color $i$, for some $i \in \{1,2\}$; otherwise, we call $G$ \emph{$(H_1,H_2)$-free}. Many classical problems in extremal graph theory have multicolor extensions. Take Tur\'an type problem as an example. It is known that every sufficiently large \((H_1,H_2)\)-free graph has edge density at most $1-\frac{1}{R(\chi(H_1),\chi(H_2))-1}+o(1)$. Such multicolor extremal questions also appear in Ramsey--Tur\'an theory; see, for instance, \cite{kim2019two}.

In analogy with the chromatic threshold, it is natural to ask how large the minimum-degree density must be before all $(H_1,H_2)$-free graphs have bounded chromatic number. Formally, we define the chromatic threshold of a pair of graphs $(H_1,H_2)$ as
\begin{align*}
\delta_\chi(H_1,H_2):=
     \inf \{d: \exists~C&=C(H_1,H_2, d)~\text{ s.t. $\forall$ $n$-vertex $(H_1,H_2)$-free graph $G$}, \\
    & \text{ if $\delta(G) \geq d n$, then $\chi(G) \leq C$}\}.
\end{align*}

This parameter provides a natural hybrid of extremal graph theory and Ramsey theory. Already in the one-color case, the value of $\delta_\chi(H)$ depends on structural properties of $H$ beyond its chromatic number. In the two-color setting, the situation is further influenced by the structure of Ramsey colorings avoiding $H_1$ and $H_2$. The difficulty of this problem is closely tied to the complexity of Ramsey theory. Even for cliques, many Ramsey numbers are unknown, making a general characterization of $\delta_\chi(H_1,H_2)$ highly challenging. However, when both $H_1$ and $H_2$ are 3-chromatic graphs, the small Ramsey number $R(3,3)=6$ gives us hope for a precise structural analysis. 

For a start, the above multicolor Tur\'an theorem shows that the pair $(K_3,K_3)$ has the same Tur\'an density as $K_6=K_{R(3,3)}$. Could it be that $\delta_\chi(K_3,K_3)=\delta_\chi(K_6)$? As we shall see, this is not the case and the full picture of $\delta_\chi(H_1,H_2)$ for 3-chromatic graphs is rather complicated. 

Our main result provides a complete classification for pairs of 3-chromatic graphs, which splits into six distinct structural regimes and yields five possible values. To state the classification, we introduce an embeddability parameter. Let $\C C_5$ denote the family of all blowups of 5-cycles. Starting from these blowups, we define three larger families $\C C_5 \subseteq \C Z_2 \subseteq \C Z_1 \subseteq \C Z_0.$
Informally, $\C Z_0,\C Z_1$ and $\C Z_2$ are obtained from a blowup of $C_5$ by replacing some parts with high-chromatic Zykov-type graphs. For a 3-chromatic graph $H$, its embeddability $\mathrm{emb}(H)\in\{0,1,2,3,4\}$ measures how easy it is to embed $H$; the higher the easier: $\mathrm{emb}(H)=4$ if it embeds into $\C C_5$, while $\mathrm{emb}(H)=0$ means it cannot be embedded into $\C Z_0$. The precise definitions are given in Section~\ref{subsec: emb}. We write $\mathrm{emb}(H_1,H_2)=(\le a,\le b)$ if $\mathrm{emb}(H_1)\le a$ and $\mathrm{emb}(H_2)\le b$.

\begin{theorem}[Classification of pairs by embeddability]
\label{thm:main}
Let \(H_1,H_2\) be \(3\)-chromatic graphs with $\operatorname{emb}(H_1)\le \operatorname{emb}(H_2)$. Then 
\[
\delta_\chi(H_1,H_2)\in
\left\{
\frac23,\frac57,\frac34,\frac79,\frac45
\right\}.
\]
More precisely, set 
$\eta:=\max\{\delta_\chi(H_1),\delta_\chi(H_2)\}$, then:
\[
\delta_\chi(H_1,H_2)=
\begin{cases}
4/5,
    & \text{if } \operatorname{emb}(H_1,H_2)
      =(0,\le 3),\\
7/9,
    & \text{if } \operatorname{emb}(H_1,H_2)
      =(1,\le3),\\
3/4,
    & \text{if } \operatorname{emb}(H_1,H_2)
      =(2,\le 3),\\
3/4,
    & \text{if } \operatorname{emb}(H_1,H_2)
      \in \{(3,3),(\le 4,4)\}
      \text{ and } \eta=1/2,\\
5/7,
    & \text{if } \operatorname{emb}(H_1,H_2)
      \in \{(3,3),(\le 4,4)\}
      \text{ and } \eta=1/3,\\
2/3,
    & \text{if } \operatorname{emb}(H_1,H_2)
      \in \{(3,3),(\le 4,4)\}
      \text{ and } \eta=0.
\end{cases}
\]
\end{theorem}

Our theorem shows that the two-color chromatic threshold is not determined only by the chromatic number of $H_i$ but also by their embeddabilities into the hierarchy of $C_5$-type Ramsey configurations. 
The decision diagram is depicted in Figure~\ref{Figure:main0}.

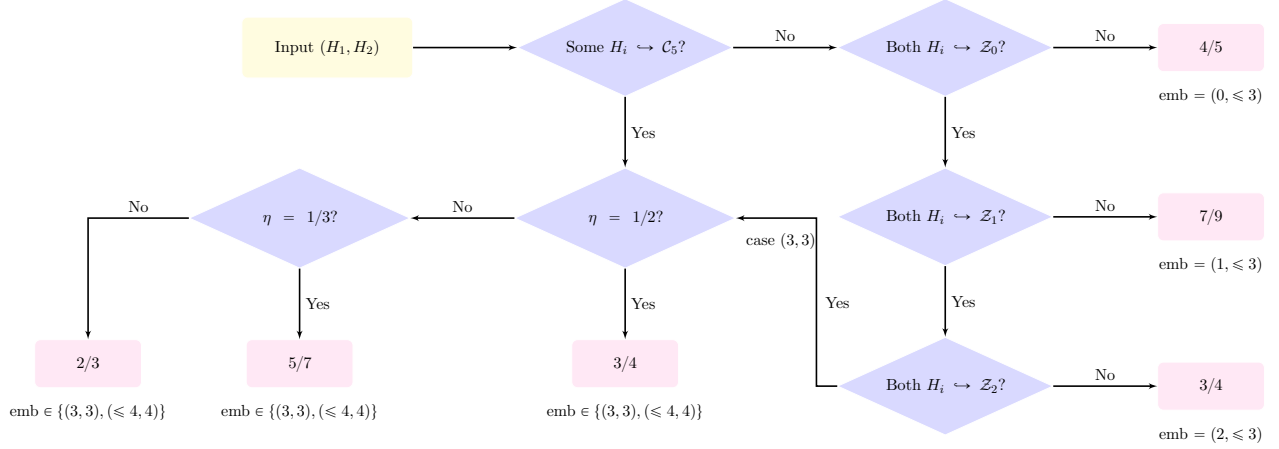
\begin{figure}[!ht]
\centering
\resizebox{1\textwidth}{!}{%
\begin{tikzpicture}[
    node distance=1.7cm and 2.5cm,
    decision/.style={
        diamond,
        aspect=2.2,
        draw=black!0,
        fill=blue!15,
        text width=11em,
        align=center,
        inner sep=1pt,
        minimum height=4.8em
    },
    block/.style={
        rectangle,
        draw=black!0,
        fill=mypink!12,
        text width=6em,
        align=center,
        rounded corners,
        minimum height=3em
    },
    start/.style={
        rectangle,
        draw=black!0,
        fill=yellow!15,
        text width=10em,
        align=center,
        rounded corners,
        minimum height=3.8em
    },
    line/.style={draw, black, line width=1.1pt, -latex'}
]

\node[start] (input) {Input $(H_1,H_2)$};

\node[decision, right=of input] (cc5) {Some $H_i\hookrightarrow \C C_5$?};

\node[decision, below=of cc5] (eta12) {$\eta=1/2$?};
\node[decision, left=of eta12] (eta13) {$\eta=1/3$?};

\node[block, below=of eta12] (val34a) {$3/4$};
\node[block, below=of eta13] (val57) {$5/7$};
\node[block, left=of val57] (val23) {$2/3$};

\node[decision, right=of cc5] (z0) {Both $H_i\hookrightarrow \C Z_0$?};
\node[decision, below=of z0] (z1) {Both $H_i\hookrightarrow \C Z_1$?};
\node[decision, below=of z1] (z2) {Both $H_i\hookrightarrow \C Z_2$?};

\node[block, right=of z0] (val45) {$4/5$};
\node[block, right=of z1] (val79) {$7/9$};
\node[block, right=of z2] (val34b) {$3/4$};

\path[line] (input) -- (cc5);

\path[line] (cc5) -- node[right] {Yes} (eta12);
\path[line] (cc5) -- node[above] {No} (z0);

\path[line] (eta12) -- node[right] {Yes} (val34a);
\path[line] (eta12) -- node[above] {No} (eta13);
\path[line] (eta13) -- node[right] {Yes} (val57);
\path[line] (eta13) -| node[near start, above] {No} (val23);

\path[line] (z0) -- node[above] {No} (val45);
\path[line] (z0) -- node[right] {Yes} (z1);
\path[line] (z1) -- node[above] {No} (val79);
\path[line] (z1) -- node[right] {Yes} (z2);
\path[line] (z2) -- node[above] {No} (val34b);

% \path[line] (z2.west) -| node[pos=0.25, below] {Yes} node[pos=0.75, right] {case $(3,3)$} (eta12.east);
\path[line] (z2.west)
    -- ++(-0.5cm,0)
    |- node[pos=0.25, right=2pt] {Yes}
       node[pos=0.72, below=6pt] {case $(3,3)$}
       (eta12.east);

\node[below=0.25cm of val23] {emb \(\in \{(3,3),(\le 4,4)\}\)};
\node[below=0.25cm of val57] {emb \(\in \{(3,3),(\le 4,4)\}\)};
\node[below=0.25cm of val34a] {emb \(\in \{(3,3),(\le 4,4)\}\)};
\node[below=0.25cm of val45] {emb \(=(0,\le 3)\)};
\node[below=0.25cm of val79] {emb \(=(1,\le 3)\)};
\node[below=0.25cm of val34b] {emb \(=(2,\le 3)\)};

\end{tikzpicture}
}
\caption{
Decision diagram for $\delta_\chi(H_1,H_2)$ when $\chi(H_i)=3$. Here $\eta=\max\{\delta_\chi(H_1),\delta_\chi(H_2)\}$.
The right-hand branch applies when neither graph embeds into $\C C_5$; the left-hand branch applies when at least one graph embeds into $\C C_5$, or when both graphs have embeddability $3$.
}
\label{Figure:main0}
\end{figure}

Even in the symmetric case $H_1=H_2=H$, the classification is already rich. The following table of examples illustrates the possible behaviors. We refer the reader to the end of Section~\ref{subsec: emb} for examples of all six structural behaviors in~\cref{thm:main} for general pairs $(H_1,H_2)$.

\begin{table}[htbp]
\centering
\renewcommand{\arraystretch}{1.15}
\setlength{\tabcolsep}{8pt}
\begin{tabular}{c c c}
\toprule
Graph $H$ & Structural reason & $\delta_\chi(H,H)$ \\
\midrule
$K_3[2]$ & $\mathrm{emb}(H)=0$ & $4/5$ \\
$C_5^+[2]$ & $\mathrm{emb}(H)=1$ & $7/9$ \\
$K_3$ & $\mathrm{emb}(H)=2$ & $3/4$ \\
$C_5[2]$ & $\mathrm{emb}(H)\in\{3,4\}$ and $\delta_\chi(H)=1/2$ & $3/4$ \\
$C^*_{10}[3]$ & $\mathrm{emb}(H)\in\{3,4\}$ and $\delta_\chi(H)=1/3$ & $5/7$ \\
$C_5$, Petersen graph & $\mathrm{emb}(H)\in\{3,4\}$ and $\delta_\chi(H)=0$ & $2/3$ \\
\bottomrule
\end{tabular}
\caption{Representative values in the symmetric case. The graphs $C_5^+[2]$ and $C^*_{10}[3]$ are defined in Section~\ref{subsec: emb} (see~\cref{figure: C_10}).}
\label{tab:sym-intro}
\end{table}

The main new difficulty, compared with the one-color chromatic threshold problem, is to control how high-chromatic pieces attach to the Ramsey-critical \(C_5\)-configuration. The embeddability parameter is designed precisely to encode these possible attachments.

The lower-bound constructions start from extremal red-blue colorings related to the \(C_5\)-coloring of \(K_5\) and insert high-chromatic gadgets into selected parts. The three families \(\C Z_0,\C Z_1,\C Z_2\) record the possible ways these gadgets can appear. These constructions show that each obstruction in the classification is genuine.

For the upper bound, we apply a two-color regularity lemma and study the red and blue reduced graphs, both of which must be triangle-free. We then partition vertices according to their red and blue density profiles toward the clusters. A high-chromatic density class either forces one of the configurations represented by \(\C Z_0,\C Z_1,\C Z_2\), or else the reduced graph satisfies one of a small number of extremal alternatives. The key embedding lemma converts these density patterns into actual monochromatic copies of \(H_1\) or \(H_2\).

\paragraph{Organization.} The rest of the paper is structured as follows. In Section~\ref{sec: pre}, we conduct some preparations and provide several lemmas that will be needed later. In Section~\ref{sec: con}, we explain the classification for $\delta_{\chi}(H_1,H_2)$ and give the construction of the  lower bound. In Section~\ref{sec: core}, we prove a key lemma (\cref{lmm:embed key}). In Section~\ref{sec: prof}, we prove our main theorem. Concluding remarks are given in \cref{sec:rmk}.

\section{Tools and auxiliary results}\label{sec: pre}

\paragraph{Notations.} All the graphs involved in this paper are simple graphs. For the sake of clarity of presentation, we omit floor and ceiling signs whenever they are inessential. Let $G$ be a graph. We will denote the set of vertices of $G$ by $V(G)$ and the set of edges by $E(G)$, and define $|G| := |V(G)|$ and $e(G):=|E(G)|$. For a vertex $u\in V(G)$, we use $N(u)$ to denote the set of neighbors of $u$ in $G$, and for a subset $U\subseteq V(G)$, we use $N(U)=\bigcap_{u \in U} N(u)$ to denote the set of \emph{common neighbors} of vertices in $U$. For $X\subseteq V(G)$, we write $N_X(v)=N(v)\cap X$ and $d_X(v)=|N_X(v)|$.   For any graph $G$ that is 2-edge-colored in red and blue,  we will use $G_{r}$ and $G_{b}$ to denote the subgraphs of $G$ restricted to red edges and blue edges, respectively.
Given disjoint sets $X$ and $Y$, we write $K[X,Y]$ for the complete bipartite graph with parts $X$ and $Y$.
For any ordered pair $(X,Y)$, we say $(X,Y)$ is \emph{$\alpha$-dense} for some $\alpha \in [0,1]$ if $|N(x)\cap Y|\ge \alpha |Y|$ for every $x\in X$, and we say $(X,Y)$ is \emph{$\alpha$-pair-dense} for some $\alpha \in [0,1]$ if $|N(\{x,y\})\cap Y|\ge \alpha |Y|$ for every $(x,y)\in X^2$.
Let $G,H$ be two graphs and let $V_1,\dots V_k$ be subsets of $V(G)$, $U_1,\dots, U_k$ be subsets of $V(H)$. We say $G\subseteq H$ with $(V_1,\dots V_k)\subseteq (U_1,\dots, U_k)$ if $G$ can be embedded into $H$ with $V_i$ embedded into $U_i$ for all $i\in[k]$. The \emph{$t$-blowup} of $H$, denoted by $H[t]$, is obtained by replacing each vertex of $H$ by an independent set of size $t$ and each edge by a copy of $K_{t,t}$.

\medskip

 \L uczak and Thomass\'e \cite{luczak2010coloring} defined a graph $H$ to be \emph{near-acyclic} if $\chi(H) = 3$ and $H$ admits a partition into a forest $F$ and an independent set $S$ such that every odd cycle of $H$ meets $S$ in at least two vertices.
 Equivalently, for each tree $T$ in $F$ with color classes $V_1 (T)$ and $V_2(T)$, there is no vertex of $S$ with neighbors in both $V_1 (T)$ and $V_2(T)$. For instance, all odd cycles are near-acyclic.

Given a graph $H$ with $\chi(H)=r \geq 3$, the \emph{decomposition family} $\mathcal{M}(H)$ of $H$ is the set of all bipartite graphs that can be obtained from $H$ by deleting $r-2$ color classes in some proper $r$-coloring of $H$. For example, $\mathcal{M}(K_{2,2,2})=\{C_4\}$.

First, we state the following characterization in \cite{allen2013chromatic} for graphs with chromatic number 3.

\begin{theorem}[\cite{allen2013chromatic}]\label{thm:characterization}
Let $H$ be a graph with $\chi(H)=3$.

\begin{itemize}
\item If $\mathcal{M}(H)$ contains no forest, then $\delta_\chi(H)=1/2$.
\item If $H$ is near-acyclic, then $\delta_\chi(H)=0$.
\item Otherwise, $\delta_\chi(H)=1/3$.
\end{itemize}

\end{theorem}

For any $k\in \mathbb{N}$, define $\mathcal{F}_k$ as the \emph{universal forest} with parameter $k$, i.e., the disjoint union of all non-isomorphic trees on $k$ vertices. Denote $\mathcal{F}^t_k$ the vertex-disjoint union of $t$ copies of $\mathcal{F}_k$. For a graph $G$, let $c(G)$ be the number of its components. 

We now define an important family of graphs, Zykov graphs, which are universal graphs for near-acyclic graphs.

 \begin{defn}[Zykov Graph]\label{def:k-t-Zykov}
  Let $k,t\in \mathbb{N}$ and $c=c(\mathcal{F}^t_k)$.
  Let $T_1, \ldots, T_{c}$ be the components of $\mathcal{F}^t_k$. 
  For every $j\in[c]$, let $T_j$ have bipartition $A^0_j \cup A^1_j$. 
  We define $Z_k^t= Z_k^t(A,U)$ to be the \emph{$(k,t)$-Zykov graph} on vertex set $A\cup U$, where
\begin{align*}
  A:=\bigcup_{j\in[c]} (A^0_{j} \cup A^1_{j}),\quad  U:=\{u^{i}_I: i\in[t] \text{ and } I \subseteq [c]\}
\end{align*}
and with edge set
\begin{align*}
  E(Z_k^t):=\bigcup_{j\in[c]}\Big(E(T_j) \cup \bigcup_{\substack{i\in [t] \\ j\in I\subseteq [c]}} K(\{u^{i}_I\}, A^0_j) \cup \bigcup_{\substack{i\in [t] \\ j\notin I\subseteq[c]}} K(\{u^{i}_I\}, A^1_j)\Big) .
\end{align*}
\end{defn}

For simplicity, for any $\bm{x} = (x_1,\dots, x_c)\in \{0,1\}^c$ we write 
\[
A^{\bm{x}}=\bigcup_{i=1}^c A^{x_i}_i.
\]

\begin{prop}[\cite{allen2013chromatic}]\label{fact:zykov-universal} For any graph $H$ with $\chi(H)=3$, the following are equivalent:
\begin{itemize}
\item[\rm (i)]  $H$ is near-acyclic.
\item[\rm (ii)] there exist $k,t \in \mathbb{N}$ such that $H$ is a subgraph of $Z_k^t$.
\item[\rm (iii)] $\delta_\chi(H)=0$.
\end{itemize}
\end{prop}

The following simple fact follows straightforwardly from the definition of Zykov graphs. We omit its proof.

\begin{fact}\label{fact:embde C5}
Let $G$ be the graph obtained from $Z^t_k$ on $(A,U)$ by adding a new independent set $S$ and all edges between $S$ and $U$. Then $G \subseteq Z^{t'}_{k}$ for some $t'>t$.
\end{fact}

We also need the following embedding lemma. It can be easily deduced from Lemmas 10, 24 and 25, and Propositions 26 and 36 in \cite{allen2013chromatic}. 
Readers can refer to Lemma A.3 (with $r=4$ and $(X,Y,Z_1)_{\mathrm{A.3}}=(X,Y,Z)$) in \cite{liu-shangguan-wu-xue} for a detailed proof.

\begin{lemma}\label{lmm:0-1/2}
     Let $k,t$ be two positive integers.
    For any $\delta >0$, there exist a constant $C$ and a sufficiently large integer $n$ such that
    for any  graph $G$, and three pairwise disjoint vertex sets $X,Y,Z$ in $G$ with $\chi(G[X]) \ge C$ and $|Y|=|Z|=n$, if $(X,Y)$ is $\delta$-dense and $(X,Z)$  is $\delta$-pair-dense,
    then there exist a copy of $Z^t_k(A,U)$ with $(A,U) \subseteq(X,Y)$ and a set $S \subseteq Z$ with $|S| \ge t$ in $G$ such that $G[A,S]$ is a complete bipartite graph. 
    
    In particular, for any $H$ with $\chi(H)=3$, if $\delta_{\chi}(H)=1/3$, then $H\subseteq G[X\cup Z]$ and if $\delta_{\chi}(H)=0$, then $H\subseteq G[X\cup Y]$ .
\end{lemma}

We will utilize the multicolor regularity lemma. Let $X$, $Y$ be two disjoint subsets of $V$, we define the density of the pair $(X,Y)$ as
   \begin{equation*}
     d(X,Y):=\frac{e(X,Y)}{|X|\cdot |Y|},
   \end{equation*}
   and we call a pair of vertex sets $X$ and $Y$ \emph{$\varepsilon$-regular}, if for all subsets $A\subseteq X$, $B\subseteq Y$ satisfying $|A|\geq \varepsilon|X|$, $|B|\geq \varepsilon|Y|$, we have
   \begin{equation*}
   \left|\frac{e(A,B)}{|A|\cdot|B|}-\frac{e(X,Y)}{|X|\cdot|Y|}\right|=|d(X,Y)-d(A,B)| \leq \varepsilon.
\end{equation*}
Moreover, if $(X,Y)$ is \emph{$\varepsilon$-regular} and $d(X,Y)\ge d$, the pair $(X,Y)$ is called \emph{$(\varepsilon,d)$-regular}. A partition of $V$ into $k+1$ sets $(V_0,V_{1},\ldots,V_{k})$ is called an \emph{$\varepsilon$-regular partition}, if $|V_0|<\varepsilon n$ and for all $1\leq i<j\leq k$ we have $|V_{i}|=|V_{j}|$, and all except $\varepsilon k^{2}$ of the pairs $(V_{i},V_{j})$, $1\leq i<j\leq k$, are $\varepsilon$-regular.

   We can now state the multicolor version of the Szemer\'{e}di regularity lemma.

\begin{theorem}[\cite{komlos1995szemeredi}]\label{thm:regularity}
   For any $\varepsilon>0$ and integers $r, \kappa$ there exists an $M$ such that if the edges of a graph $G$ are $2$-edge-colored in red and blue, then the vertex set $V(G)$ can be partitioned into sets $V_0, V_1, \ldots, V_k$, for some $\kappa \leq k \leq M$, so that $(V_0,V_{1},\ldots,V_{k})$ is an $\varepsilon$-regular partition with respect to both $G_{r}$ and $G_{b}$.
\end{theorem}

  Given a 2-edge-colored graph $G$ and an $\varepsilon$-regular partition $V_0 \cup V_1 \cup \cdots \cup V_k$ of $V(G)$ and $0<d<1$, we define graphs $R_r$ and $R_b$, called the \emph{$(\varepsilon, d)$-reduced graphs} of $G$, with  $V(R_r)=V(R_b)=[k]$, where for any two different vertices $i,j$, $i j \in E(R_r)$ ($E(R_b)$, respectively) if and only if $(V_i, V_j)$ is an $(\varepsilon, d)$-regular pair in $G_r$ ($G_b$, respectively).

\begin{theorem}[Embedding Lemma \cite{komlos1995szemeredi}]\label{thm:embedding-lemma}
  For all $d \in(\varepsilon,1]$ and $\Delta \geq 1$ there exists an $\varepsilon_0>0$ with the following property. 
  Suppose that a graph $G$ has an $\varepsilon$-regular partition $\{V_0, \ldots, V_k\}$ with $|V_1|=\cdots=|V_k|=m$ and the $(\varepsilon, d)$-reduced graph $R$, where $\varepsilon \leq \varepsilon_0$ and $m d^{\Delta} \geq 2 s$ for some integer $s \geq 1$. 
  Assume $H\subseteq R[s]$ is a subgraph of the $s$-blowup of $R$ with $\Delta(H) \leq \Delta$. Then $H\subseteq G$.
\end{theorem}

We state one more useful fact about subpairs of $(\varepsilon, d)$-regular pairs. 

\begin{fact}[Slicing lemma~\cite{komlos1995szemeredi}]\label{fact:slicinglemma} 
    Let $(U, W)$ be an $(\varepsilon, d)$-regular pair and suppose that $U' \subseteq U$, $W' \subseteq W$ satisfy $|U'| \geq \alpha|U|$ and $|W'| \geq \alpha|W|$. Then $(U', W')$ is $(\varepsilon / \alpha, d-\varepsilon)$-regular.
\end{fact}

\section{Characterizations and lower bound constructions}\label{sec: con}
The following lemma is the main result of this section, providing the extremal constructions that attain lower bounds in~\cref{thm:main}. We first present in~\cref{subsec: emb} the necessary definitions required to read~\cref{Figure:main0} and~\cref{prop: lower bound}. After providing some building blocks in~\cref{sec:block}, we will give the six types of extremal constructions $G^1,\ldots,G^6$ in~\cref{sec:ext-constr} for~\cref{prop: lower bound} parts 1--6 respectively.

\begin{lemma}\label{prop: lower bound}
Let $H_1,H_2$ be 3-chromatic graphs with  $\mathrm{emb}(H_1)\le \mathrm{emb}(H_2)$.
\begin{enumerate}
    \item If $H_1 \not\hookrightarrow\C Z_0$ and $H_2 \not\hookrightarrow \C C_5$, i.e.~$\mathrm{emb}(H_1,H_2)=(0,\le 3)$, then $\delta_\chi(H_1,H_2)\ge 4/5$.
    \item  If $H_1 \not\hookrightarrow\C Z_1$ and $H_2 \not\hookrightarrow \C C_5$, i.e.~$\mathrm{emb}(H_1,H_2)=( \le 1,\le 3)$, then $\delta_\chi(H_1,H_2)\ge 7/9$.
    \item If $H_1 \not\hookrightarrow\C Z_2$ and $H_2 \not\hookrightarrow \C C_5$, i.e.~$\mathrm{emb}(H_1,H_2)=(\le 2,\le 3)$, then $\delta_\chi(H_1,H_2)\ge 3/4$.
    \item If $\max\{\delta_{\chi}(H_1),\delta_{\chi}(H_2)\}=1/2$, then $\delta_\chi(H_1,H_2)\ge 3/4$.
    \item If $\max\{\delta_{\chi}(H_1),\delta_{\chi}(H_2)\}=1/3$, then $\delta_\chi(H_1,H_2)\ge 5/7$.
    \item If $\max\{\delta_{\chi}(H_1),\delta_{\chi}(H_2)\}=0$, then $\delta_\chi(H_1,H_2)\ge 2/3$. 
\end{enumerate}
\end{lemma} 

%See Figure~\ref{Figure:main0} for a summary of the case distinction.

\subsection{How to read the manual}\label{subsec: emb}
In this subsection, we introduce the notions needed in~\cref{prop: lower bound} and explain how to read the characterization in~\cref{Figure:main0}.

To formalize the notion of embeddability, we first define three relevant families of graphs obtained from merging Zykov graphs and blowups of 5-cycles. The embeddability of $H_1$ and $H_2$ into these families of graphs forms the basis of our classification. Let $\C C_5=\bigcup_{t\in\mathbb{N}}C_5[t]$ be the collection of all blowups of $C_5$.

\begin{defn}[Families $\C Z_0$, $\C Z_1$ and $\C Z_2$, see~\cref{C_5(F)}]
Let $s=\max\{|\mathcal{F}_k^t|,t2^{c(\mathcal{F}_k^t)}\}$ and $G=C_5[s]$ be an $s$-blowup of $C_5$ on vertex sets $(V_1,\ldots, V_5)$ where edges are between $(V_i,V_j)$ when $j-i\equiv \pm1 \pmod{5}$. 
\begin{itemize}
    \item Let $Z_0^{k,t}$ be the graph obtained from $G$ by inserting a copy of $\mathcal{F}_k^t$ into $V_1$.
    
    \item  Let $Z_1^{k,t}$ be the graph obtained from $G$ by replacing $G[V_1,V_2]$ with a copy of $Z_k^t(A,U)$.\footnote{Here, formally we first delete all edges between $V_1$ and $V_2$. Then choose subsets $A\subseteq V_1$ and $U\subseteq V_2$, and identify $A\cup U$ with the vertex set of $Z_k^t(A,U)$ so that $A$ (resp.\ $U$) plays the role of the $A$-part (resp.\ $U$-part) of $Z_k^t$. Finally, add all edges of $Z_k^t$ between $A$ and $U$ (and inside $A$) accordingly, while leaving all other edges of $G$ unchanged.}
    \item  Let $Z_2^{k,t}$ be the graph obtained from $G$ by replacing each of $G[V_1,V_2]$ and $G[V_1,V_5]$ with a copy of $Z_k^t(A,U)$ such that the two images of $A$ coincide.
\end{itemize}
For $i\in\{0,1,2\}$, let $\C Z_i=\bigcup_{k,t\in\mathbb{N}}Z_i^{k,t}$.
\end{defn}

\begin{figure}[!ht]
    \centering
    \resizebox{1\textwidth}{!}{%
\begin{tikzpicture}
\coordinate (r1'') at (3-3,0);
\coordinate (r2'') at (4-3,0);
\coordinate (r3'') at (5.31-3,-0.95);
\coordinate (r4'') at (5.62-3,-1.9);
\coordinate (r5'') at (5.12-3,-3.44);
\coordinate (r6'') at (3.9-3,-4.3);
\coordinate (r7'') at (3.1-3,-4.3);
\coordinate (r8'') at (1.88-3,-3.44);
\coordinate (r9'') at (1.38-3,-1.9);
\coordinate (r10'') at (1.69-3,-0.95);
\draw [red,line width=2pt](r1'') -- (r2'');
\draw [red,line width=2pt](r3'') -- (r4'');
\draw [red,line width=2pt](r5'') -- (r6'');
\draw [red,line width=2pt](r7'') -- (r8'');
\draw [red,line width=2pt](r9'') -- (r10'');

\draw (2-3,0) circle (1cm);
\draw (5-3,0) circle (1cm);
\draw (1.07-3,-2.85) circle (1cm);
\draw (5.93-3,-2.85) circle (1cm);
\draw (3.5-3,-4.6) circle (0.5cm);

\node[below] at (10.5-10,-4.3) {\textcolor{red}{$\mathcal{F}_k^t$}};
\node[below] at (10.5-10,-5.3) {\textcolor{red}{\huge{$\C Z_0$}}};
\node[below] at (8.1-10,-2.6) {\textcolor{red}{$V_2$}};
\node[below] at (9-10,0.2) {\textcolor{red}{$V_3$}};
\node[below] at (12-10,0.2) {\textcolor{red}{$V_4$}};
\node[below] at (13-10,-2.6) {\textcolor{red}{$V_5$}};
%%%%%%%%%%%%%%%%%%%%%%%%%%%%%%%%%%%%%%%%%%%%%%%%%%%%%%%%%%%%%%%%
\coordinate (r1') at (3+7,0);
\coordinate (r2') at (4+7,0);
\coordinate (r3') at (5.31+7,-0.95);
\coordinate (r4') at (5.62+7,-1.9);
\coordinate (r5') at (5.12+7,-3.44);
\coordinate (r6') at (3.9+7,-4.3);
\coordinate (r7') at (3.1+7,-4.3);
\coordinate (r8') at (1.88+7,-3.44);
\coordinate (r9') at (1.38+7,-1.9);
\coordinate (r10') at (1.69+7,-0.95);
\draw [red,line width=2pt](r1') -- (r2');
\draw [red,line width=2pt](r3') -- (r4');
\draw [red,line width=2pt](r5') -- (r6');
\draw [dashed,red,line width=2pt](r7') -- (r8');
\draw [red,line width=2pt](r9') -- (r10');

\draw (2+7,0) circle (1cm);
\draw (5+7,0) circle (1cm);
\draw (1.07+7,-2.85) circle (1cm);
\draw (5.93+7,-2.85) circle (1cm);
\draw (3.5+7,-4.6) circle (0.5cm);

\node[below] at (10.5,-4.3) {\textcolor{red}{$\mathcal{F}_k^t$}};
\node[below] at (10.5,-5.3) {\textcolor{red}{\huge{$\C Z_1$}}};
\node[below] at (8.1,-2.6) {\textcolor{red}{$U$}};
\node[below] at (9,0.2) {\textcolor{red}{$V_3$}};
\node[below] at (12,0.2) {\textcolor{red}{$V_4$}};
\node[below] at (13,-2.6) {\textcolor{red}{$V_5$}};
%%%%%%%%%%%%%%%%%%%%%%%%%%%%%%%%%%%%%%%%%%%%%%%%%%%%%%%%%%%%%%%%
\coordinate (r1) at (3+17,0);
\coordinate (r2) at (4+17,0);
\coordinate (r3) at (5.31+17,-0.95);
\coordinate (r4) at (5.62+17,-1.9);
\coordinate (r5) at (5.12+17,-3.44);
\coordinate (r6) at (3.9+17,-4.3);
\coordinate (r7) at (3.1+17,-4.3);
%\coordinate (r6) at (3.95+17,-4.1);
%\coordinate (r7) at (3.05+17,-4.1);
\coordinate (r8) at (1.88+17,-3.44);
\coordinate (r9) at (1.38+17,-1.9);
\coordinate (r10) at (1.69+17,-0.95);
\draw [red,line width=2pt](r1) -- (r2);
\draw [red,line width=2pt](r3) -- (r4);
\draw [dashed,red,line width=2pt](r5) -- (r6);
\draw [dashed,red,line width=2pt](r7) -- (r8);
\draw [red,line width=2pt](r9) -- (r10);

\draw (2+17,0) circle (1cm);
\draw (5+17,0) circle (1cm);
\draw (1.07+17,-2.85) circle (1cm);
\draw (5.93+17,-2.85) circle (1cm);
\draw (3.5+17,-4.6) circle (0.5cm);

\node[below] at (3.5+17,-4.3) {\textcolor{red}{$\mathcal{F}_k^t$}};
\node[below] at (3.5+17,-5.3) {\textcolor{red}{\huge{$\C Z_2$}}};
\node[below] at (3.5+17-2.4,-2.6) {\textcolor{red}{$U$}};
\node[below] at (3.5+17-2.4+4.9,-2.6) {\textcolor{red}{$U'$}};
\node[below] at (3.5+17-2.4+1,0.2) {\textcolor{red}{$V_3$}};
\node[below] at (3.5+17-2.4+4,0.2) {\textcolor{red}{$V_4$}};
\end{tikzpicture}
}
\caption{Families $\C Z_0$, $\C Z_1$ and $\C Z_2$. One can think of the subscript $i\in\{0,1,2\}$ in $\C Z_i$ as the number of Zykov graphs used to replace blowup of edges in $\C C_5$ (which is the number of dashed edges in the figure).}\label{C_5(F)}
\end{figure}
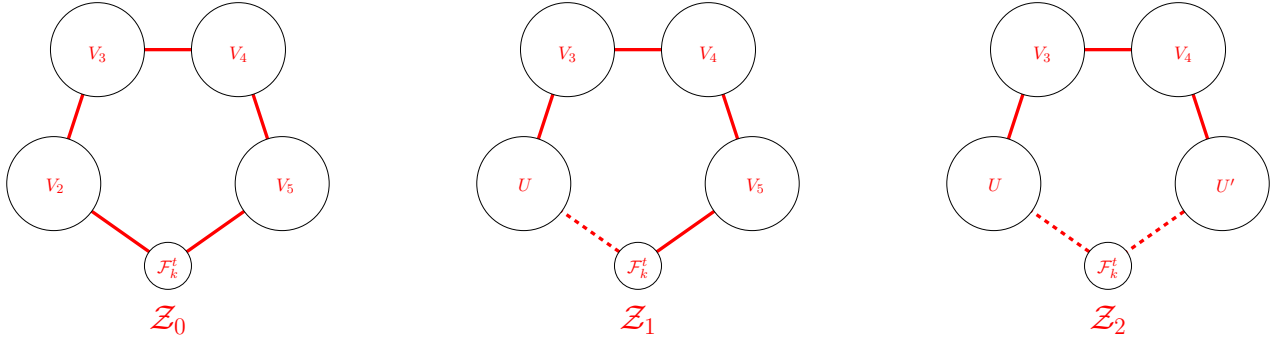

Given a graph $H$ and a family of graphs $\C H$, we say $H$ \emph{embeds into} $\C H$, denoted by $H\hookrightarrow \C H$, if $H$ is a subgraph of some graph in $\C H$. It is not hard to  see that  
$$H \hookrightarrow \C C_5 \quad \implies\quad H \hookrightarrow \C Z_2 \quad\implies\quad H \hookrightarrow \C Z_1 \quad\implies\quad H \hookrightarrow \C Z_0.$$ 
We can now define the embeddability of $H$ depending on where in the chain above it lies. Roughly speaking, the higher the embeddability of $H$ is, the easier it is to embed $H$.
\begin{defn}[Embeddability] 
For a graph 
$H$ with 
$\chi(H)=3$, the \emph{embeddability} $\mathrm{emb}(H)$ is defined as follows:
\begin{itemize}
    \item $\mathrm{emb}(H)=0$ if $H\not\hookrightarrow \C Z_0$;
    \item $\mathrm{emb}(H)=1$ if $H\not\hookrightarrow \C Z_1$ but $H\hookrightarrow \C Z_0$;
    \item $\mathrm{emb}(H)=2$ if $H\not\hookrightarrow \C Z_2$  but $H\hookrightarrow \C Z_1$;
    \item $\mathrm{emb}(H)=3$ if $H\not\hookrightarrow \C C_5$  but $H\hookrightarrow \C Z_2$;
    \item $\mathrm{emb}(H)=4$ if $H\hookrightarrow \C C_5$.
\end{itemize}
\end{defn}

\paragraph{Examples.} 
The following table demonstrates that although
embeddability and chromatic threshold are related, they capture different structural properties. The embeddability of each graph in the following table can be easily verified from the definitions. Similarly, the chromatic threshold is straightforward to determine for all cases except for one graph ($C^*_{10}[3]$), which we will prove later.
The symbol ``$/$" in the following table indicates that no such graph exists.
The symbol ``$\cup$" denotes the disjoint union of graphs.
\[
\begin{array}{c|ccccc}
H  &\mathrm{emb}(H) =4 & \mathrm{emb}(H) =3 & \mathrm{emb}(H) =2 & \mathrm{emb}(H) =1 &\mathrm{emb}(H) =0 \\ \hline
\delta_{\chi}(H) =0 & C_5 & PG & / & / & /\\
\delta_{\chi}(H) =1/3 & C^*_{10}[3]  &PG\cup C^*_{10}[3] & K_3 & /& /\\ 
\delta_{\chi}(H) =1/2  & C_5[2] & PG\cup C_5[2] & C_5[2] \cup K_3 & C^+_5[2] & K_3[2] \\
\end{array}
\]
In this table, $PG$ is the Petersen graph, $K_3[2]$ denotes the 2-blowup of the triangle $K_3$, $C_5[2]$ denotes the 2-blowup of the cycle $C_5$, and $C^+_5[2]$ is obtained from $C_5[2]$ by adding one edge within one of the blowup classes.
Let $C^*_{10}[3]$ be the graph obtained as follows. Start with a 3-blowup of the cycle $C_{10}$, in which the vertices (in order around the cycle) are replaced by partite sets
\[
(X_1, Y_2, X_3, Y_4, X_5, Y_1, X_2, Y_3, X_4, Y_5).
\]
Additionally, we add a perfect matching between the following pairs: $(Y_2, Y_3),(Y_3, Y_4),(Y_4, Y_5)$ (See Figure~\ref{figure: C_10}).

\begin{figure}[ht] % 使用figure环境
    \centering
\begin{minipage}{0.4\textwidth}
    \centering
\resizebox{0.8\textwidth}{!}{
\begin{tikzpicture}[scale=1.5,
    every node/.style={
        circle, 
        draw=black, % 设置边框颜色为黑色
        fill=black, % 设置填充颜色为白色
        minimum size=30pt, % 设置节点的最小尺寸
        inner sep=0pt % 设置内边距为0
    }
]
    % 定义点
    \node (A1) at (0,-16) {};
    \node (A2) at (0,-19) {};
    \node (A3) at (0,-22) {};

    \node (B1) at (10,0) {};
    \node (B2) at (12,0) {};
    \node (B3) at (14,0) {};

    \node (C1) at (11.5,17.5) {};
    \node (C2) at (13,19) {};
    \node (C3) at (14.5,21) {};

    \node (D1) at (-5,8) {};
    \node (D2) at (-6,9.3) {};
    \node (D3) at (-7,10.5) {};

    \node (E1) at (-18,0) {};
    \node (E2) at (-21,0) {};
    \node (E3) at (-24,0) {};

    \node (F1) at (0,-8) {};
    \node (F2) at (0,-10) {};
    \node (F3) at (0,-12) {};

    \node (G1) at (18,0) {};
    \node (G2) at (21,0) {};
    \node (G3) at (24,0) {};

    \node (H1) at (5,8) {};
    \node (H2) at (6,9.3) {};
    \node (H3) at (7,10.5) {};

    \node (I1) at (-11.5,17.5) {};
    \node (I2) at (-13,19) {};
    \node (I3) at (-14.5,21) {};

    \node (J1) at (-10,0) {};
    \node (J2) at (-12,0) {};
    \node (J3) at (-14,0) {};
    
    \draw[blue, thick] (A2) ellipse [x radius=3 , y radius=4];
    \draw[blue, thick] (B2) ellipse [x radius=3, y radius=4];
    \draw[blue, thick] (C2) ellipse [x radius=3, y radius=4];
    \draw[blue, thick] (D2) ellipse [x radius=3, y radius=4];
    \draw[blue, thick] (E2) ellipse [x radius=4, y radius=3];

    \draw[blue, thick] (F2) ellipse [x radius=3, y radius=4];
    \draw[blue, thick] (G2) ellipse [x radius=4, y radius=3];
    \draw[blue, thick] (H2) ellipse [x radius=3, y radius=4];
    \draw[blue, thick] (I2) ellipse [x radius=3, y radius=4];
    \draw[blue, thick] (J2) ellipse [x radius=3, y radius=4];

    \draw[black] (2,-16) -- (10,-3);
    \node[draw=none, fill=none, minimum size=3pt, inner sep=0pt, text=black, scale =10] at (5,-20) {$X_1$};
    \draw[black] (12,4) -- (13,15);
    \node[draw=none, fill=none, minimum size=3pt, inner sep=0pt, text=black, scale =10] at (7,0) {$Y_2$};
    \draw[black] (11,16) -- (-4,12.3);
    \node[draw=none, fill=none, minimum size=3pt, inner sep=0pt, text=black, scale =10] at (18,18) {$X_3$};
    \draw[black] (-9,9.3)-- (-18,2);
    \node[draw=none, fill=none, minimum size=3pt, inner sep=0pt, text=black, scale =10] at (-3,3) {$Y_4$};
    \draw[black] (-21,-3) -- (-3,-10);
    \node[draw=none, fill=none, minimum size=3pt, inner sep=0pt, text=black, scale =10] at (-28,0) {$X_5$};
    \draw[black] (3,-10) --(21,-3);
    \node[draw=none, fill=none, minimum size=3pt, inner sep=0pt, text=black, scale =10] at (0,-4) {$Y_1$};    
    \draw[black] (21,3) -- (9,9.3);
    \node[draw=none, fill=none, minimum size=3pt, inner sep=0pt, text=black, scale =10] at (28,0) {$X_2$};
    \draw[black] (4,12.3) -- (-10,19);
    \node[draw=none, fill=none, minimum size=3pt, inner sep=0pt, text=black, scale =10] at (3,3) {$Y_3$};   
    \draw[black] (-13,15) -- (-12,4); 
    \node[draw=none, fill=none, minimum size=3pt, inner sep=0pt, text=black, scale =10] at (-18,18) {$X_4$};
    \draw[black] (-12,-4) -- (-2,-16);
    \node[draw=none, fill=none, minimum size=3pt, inner sep=0pt, text=black, scale =10] at (-6,0) {$Y_5$};
    
    \draw[red,line width=5pt] (B1) -- (H1);
    \draw[red,line width=5pt] (B2) -- (H2);
    \draw[red,line width=5pt] (B3) -- (H3);
    \draw[red,line width=5pt] (D1) -- (H1);
    \draw[red,line width=5pt] (D2) -- (H2);
    \draw[red,line width=5pt] (D3) -- (H3);
    \draw[red,line width=5pt] (D1) -- (J1);
    \draw[red,line width=5pt] (D2) -- (J2);
    \draw[red,line width=5pt] (D3) -- (J3);
\end{tikzpicture}
}
\caption{$C_{10}^*[3]$}\label{figure: C_10}
\end{minipage}\hfill
\begin{minipage}{0.4\textwidth}
    \centering
\resizebox{0.6\textwidth}{!}{
\begin{tikzpicture}
  % 参数
  \def\Rbig{2}     % 大五边形外接圆半径
  \def\Rsmall{1}   % 小五边形外接圆半径
  \def\n{5}        % 五边形
  \def\offset{90}  % 让最下边与水平线平行（正多边形旋转角度）

  % 画大五边形及其顶点
  \foreach \i in {0,...,\numexpr\n-1} {
      \node[circle,draw,fill=white,minimum size=4pt,inner sep=0pt] (B\i) at
        ({\offset + \i*360/\n}:\Rbig) {};
  }
  \draw (B0)--(B1)--(B2)--(B3)--(B4)--cycle;
  \draw (B4)--(B0);
  
  % 画小五边形及其顶点
  \foreach \i in {0,...,\numexpr\n-1} {
      \node[circle,draw,fill=white,minimum size=4pt,inner sep=0pt] (S\i) at
        ({\offset + \i*360/\n}:\Rsmall) {};
  }
  \draw (S0)--(S2);
  \draw (S1)--(S3);
  \draw (S2)--(S4);
  \draw (S3)--(S0);
  \draw (S4)--(S1);
  \draw (S0)--(B0);
  \draw (S1)--(B1);
  \draw (S2)--(B2);
  \draw (S3)--(B3);
  \draw (S4)--(B4);
\end{tikzpicture}
}
\caption{Petersen graph}
\end{minipage}
\end{figure}

\begin{prop}
    \(\delta_{\chi}(C^*_{10}[3])=1/3\).
\end{prop}

\begin{proof}
Let \(H=C^*_{10}[3]\), \(X=\bigcup_{i=1}^5 X_i\), and
\(Y=\bigcup_{i=1}^5 Y_i\), with subscripts taken modulo \(5\). The graph
\(H\) has a proper \(3\)-coloring in which \(X\) is one color class and
\(H[Y]\) is colored with the other two colors. Since \(H\) contains an
odd cycle, this is an optimal coloring. Deleting the color class \(X\)
leaves \(H[Y]\), which is a forest. Hence \(\mathcal M(H)\) contains a
forest.

It remains to show that \(H\) is not near-acyclic. Suppose otherwise. Then
there is an independent set \(S\subseteq V(H)\) such that \(H-S\) is a
forest and every odd cycle of \(H\) meets \(S\) in at least two vertices.

For \(i\in[5]\), write \(A_i\) for the statement \(|S\cap X_i|\ge 2\), and
\(B_i\) for the statement \(|S\cap Y_i|\ge 2\). Since \(H[X_i,Y_{i+1}]\cong
K_{3,3}\), the statements \(A_i\) and \(B_{i+1}\) cannot both hold, by
independence of \(S\). They also cannot both fail, since then \(H-S\) would
retain a copy of \(K_{2,2}\) inside \(H[X_i,Y_{i+1}]\), giving a cycle.
Thus \(A_i\) holds exactly when \(B_{i+1}\) fails. Similarly, using
\(H[Y_i,X_{i+1}]\cong K_{3,3}\), \(B_i\) holds exactly when \(A_{i+1}\)
fails.

It follows that \(B_{i+1}\) holds exactly when \(B_{i-1}\) holds. Since the
indices are modulo \(5\), either all \(B_i\) hold or none of them holds.

If all \(B_i\) hold, then \(|S\cap Y_4|\ge 2\) and \(|S\cap Y_5|\ge 2\).
But the edges between \(Y_4\) and \(Y_5\) form a matching of size \(3\), so
any two vertices in \(Y_4\) and any two vertices in \(Y_5\) contain a
matched pair. Hence \(S\) contains an edge, contradiction.

Therefore none of the \(B_i\) holds. Hence all \(A_i\) hold. In particular,
for every \(j\), the set \(S\cap X_{j-1}\) has size at least \(2\). Since
\(H[X_{j-1},Y_j]\cong K_{3,3}\), no vertex of \(Y_j\) can lie in \(S\).
Thus \(S\cap Y=\varnothing\).

Now choose \(x\in S\cap X_1\), and take one of the paths
\(y_2y_3y_4y_5\) in \(H[Y]\), with \(y_j\in Y_j\). Since \(X_1\) is
complete to both \(Y_2\) and \(Y_5\), the vertices
\(x,y_2,y_3,y_4,y_5\) form a \(5\)-cycle. This cycle meets \(S\) exactly in
the single vertex \(x\), contradicting near-acyclicity.

Thus \(H\) is not near-acyclic. By \cref{thm:characterization}, we conclude
that \(\delta_{\chi}(H)=1/3\).
\end{proof}

By~\cref{Figure:main0}, we have the following table of examples of $(H_1,H_2)$ that attain all values in~\cref{thm:main}.
\[
\begin{array}{c|cccccccccc}
\delta_\chi(\cdot,\cdot)  & C_5 & C^*_{10}[3] &C_5[2]&  PG &PG\cup C^*_{10}[3] &PG\cup C_5[2] & K_3 & C_5[2] \cup K_3 & C^+_5[2]& K_3[2]\\ \hline 
C_5 & 2/3 & 5/7 & 3/4 & 2/3 & 5/7& 3/4 & 5/7 & 3/4 &3/4 & 3/4\\

C^*_{10}[3] & / & 5/7 & 3/4 & 2/3 & 5/7& 3/4 & 5/7 & 3/4 &3/4 & 3/4\\

C_5[2] & / & / & 3/4 & 3/4 & 3/4 & 3/4 & 3/4 & 3/4 & 3/4 & 3/4\\

PG  & / & / & / & 2/3 & 5/7 & 3/4 & 3/4 &3/4 &7/9 & 4/5\\

PG \cup C^*_{10}[3] & / & / & / & / & 5/7 & 3/4 & 3/4 &3/4 &7/9 & 4/5\\

PG\cup C_5[2]  & / & /& / & / & / &3/4 & 3/4 &3/4 &7/9 & 4/5\\

K_3  & / &/ & / & / & /& /  & 3/4 &3/4 &7/9 & 4/5\\

C_5[2] \cup K_3  &  /& /&  / & / & / &/ & / &3/4 &7/9 & 4/5\\

C^+_5[2]  & / &/ & / & / &/ & /  &/ &/ &7/9 & 4/5\\

K_3[2] & / & /& / & / & /&/ & / &/ &/ & 4/5\\

\end{array}
\]

\subsection{Building blocks}\label{sec:block}
In this subsection, we present some building blocks that will be used in the extremal constructions. We first recall the extremal constructions for the chromatic thresholds of single graphs.

\paragraph{Extremal graph for $\delta_\chi(H)=1/2$.}
The \emph{girth} of a graph $G$, denoted by $g(G)$, is the length of its shortest cycle. 
For each $k, \ell \in \mathbb{N}$, a graph is a \emph{$(k, \ell)$-Erd\H{o}s graph} if it has chromatic number at least $k$, and girth at least $\ell$. Erd\H{o}s \cite{erdos1959graph} showed the existence of \emph{$(k, \ell)$-Erd\H{o}s graph} for all $n\ge n_0(k,\ell)$.

  Suppose $H$ is a graph with $\chi(H)=3$ such that $\mathcal{M}(H)$ contains no forest.
  The following construction shows that $\delta_\chi(H)\ge 1/2$.
  Let $k \in \mathbb{N}$, and let $G'$ be a $(k,|H|+1)$-Erd\H{o}s graph.
  Let $G$ be the graph obtained from a complete bipartite graph $K_{|G'|,|G'|}$ by replacing one of its parts with a copy of $G'$. 
  Then $\delta(G)\ge|G|/2$ and $\chi(G) \geq k$. 
As $\C M(H)$ has no forest, $G$ is $H$-free.

\paragraph{Extremal graph for $\delta_\chi(H)=1/3$.}
The following theorem provides an extremal construction showing that $\delta_\chi(H)\ge 1/3$ for 3-chromatic $H$ that is not near-acyclic.  
 
\begin{theorem}[\L uczak and Thomass\'e \cite{luczak2010coloring}]\label{LT}
  For every $k, \ell \in \mathbb{N}$, a real number $\alpha>0$ and $n \ge n_0(k,\ell,\alpha)$, there exists an $n$-vertex $(k,\ell,\alpha)$-Borsuk-Hajnal graph $\mathrm{BH}=\mathrm{BH}(n,k,\ell)$ on vertex set $U' \cup W \cup X$ satisfying the following:
  \begin{itemize}
      \item $\chi(\mathrm{BH}[U']) \geq k, \quad   g(\mathrm{BH}[U']) \geq \ell, \quad  \text { and } \quad \delta(\mathrm{BH}) \geq (1/3-\alpha) n$;
      \item $|U'| \le \alpha n, \quad  \text { and } \quad  \forall u\in U'$, $|N(u)\cap W| \ge (1/2-\alpha)|W|$;
      
      \item $|W|=2|X|$, $\forall x\in X$, $N(x)=W$, and $W$ is an independent set;
 
      \item Every subgraph $H \subseteq \mathrm{BH}$ with $|H|<\ell$ and $\chi(H)=3$ is near-acyclic, i.e, $\delta_{\chi}(H)=0$.
  \end{itemize}
\end{theorem}

Roughly speaking, the part $\mathrm{BH}[U']$ is a variant of Borsuk graph with high girth and chromatic number where vertices are uniformly distributed on a high-dimensional sphere and two vertices are adjacent if they are at almost antipodal positions. We refer the reader to \cite{allen2013chromatic} for the details of the construction.

We can now define two families of building blocks.

\begin{defn}[Families $B_i(n,k,\ell)$, see~\cref{borsuk}]
Let $\mathrm{BH}^-$ be the graph on vertex set $U'\cup W$ obtained from $\mathrm{BH}(m,k,\ell)$ by removing the vertex set $X$. 
\begin{itemize}
    \item Let $B_1(n,k,\ell)$ be the $n$-vertex graph obtained from $\mathrm{BH}^-$ by adding vertex sets $Z_1$, $Z_2$, $Z_3$ with $|Z_1|=|Z_2|=|Z_3|=|W|$ and adding all edges in $K[W,Z_1]\cup K[Z_1,Z_2]\cup K[Z_2,Z_3]\cup K[U',Z_3]$.\footnote{In both $B_i(n,k,\ell)$, we should think of $\alpha=o(1)$ and so we  will omit mentioning $\alpha$. Thus, $|U'|=o(m)$, $|W|=(\frac{2}{3}+o(1))m$ and $n=(\frac{8}{3}+o(1))m$.}

    \item Let $B_2(n,k,\ell)$ be the $n$-vertex graph obtained from $B_1(n,k,\ell)$ by replacing the part induced on $U'\cup Z_3$ with a copy of $\mathrm{BH}^-$ such that the two images of $U'$ are identified.
\end{itemize}
\end{defn}

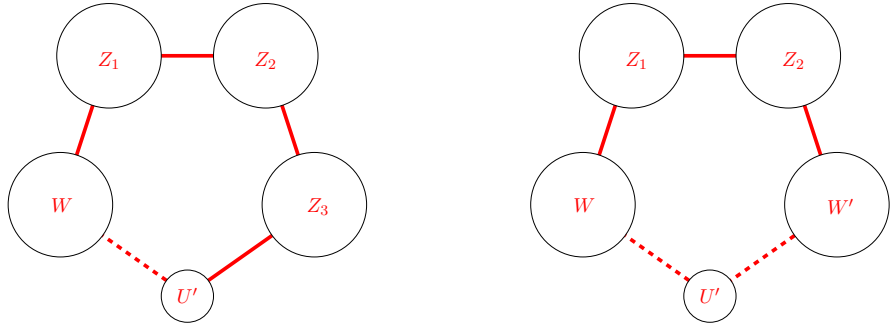
\begin{figure}[!ht]
    \centering
    \resizebox{0.7\textwidth}{!}{%
\begin{tikzpicture}
\coordinate (r1') at (3+7,0);
\coordinate (r2') at (4+7,0);
\coordinate (r3') at (5.31+7,-0.95);
\coordinate (r4') at (5.62+7,-1.9);
\coordinate (r5') at (5.12+7,-3.44);
\coordinate (r6') at (3.9+7,-4.3);
\coordinate (r7') at (3.1+7,-4.3);
\coordinate (r8') at (1.88+7,-3.44);
\coordinate (r9') at (1.38+7,-1.9);
\coordinate (r10') at (1.69+7,-0.95);
\draw [red,line width=2pt](r1') -- (r2');
\draw [red,line width=2pt](r3') -- (r4');
\draw [red,line width=2pt](r5') -- (r6');
\draw [dashed,red,line width=2pt](r7') -- (r8');
\draw [red,line width=2pt](r9') -- (r10');

\draw (2+7,0) circle (1cm);
\draw (5+7,0) circle (1cm);
\draw (1.07+7,-2.85) circle (1cm);
\draw (5.93+7,-2.85) circle (1cm);
\draw (3.5+7,-4.6) circle (0.5cm);

\node[below] at (10.5,-4.3) {\textcolor{red}{$U'$}};
\node[below] at (8.1,-2.6) {\textcolor{red}{$W$}};
\node[below] at (9,0.2) {\textcolor{red}{$Z_1$}};
\node[below] at (12,0.2) {\textcolor{red}{$Z_2$}};
\node[below] at (13,-2.6) {\textcolor{red}{$Z_3$}};
%%%%%%%%%%%%%%%%%%%%%%%%%%%%%%%%%%%%%%%%%%%%%%%%%%%%%%%%%%%%%%%%
\coordinate (r1) at (3+17,0);
\coordinate (r2) at (4+17,0);
\coordinate (r3) at (5.31+17,-0.95);
\coordinate (r4) at (5.62+17,-1.9);
\coordinate (r5) at (5.12+17,-3.44);
\coordinate (r6) at (3.9+17,-4.3);
\coordinate (r7) at (3.1+17,-4.3);
%\coordinate (r6) at (3.95+17,-4.1);
%\coordinate (r7) at (3.05+17,-4.1);
\coordinate (r8) at (1.88+17,-3.44);
\coordinate (r9) at (1.38+17,-1.9);
\coordinate (r10) at (1.69+17,-0.95);
\draw [red,line width=2pt](r1) -- (r2);
\draw [red,line width=2pt](r3) -- (r4);
\draw [dashed,red,line width=2pt](r5) -- (r6);
\draw [dashed,red,line width=2pt](r7) -- (r8);
\draw [red,line width=2pt](r9) -- (r10);

\draw (2+17,0) circle (1cm);
\draw (5+17,0) circle (1cm);
\draw (1.07+17,-2.85) circle (1cm);
\draw (5.93+17,-2.85) circle (1cm);
\draw (3.5+17,-4.6) circle (0.5cm);

\node[below] at (3.5+17,-4.3) {\textcolor{red}{$U'$}};
\node[below] at (3.5+17-2.4,-2.6) {\textcolor{red}{$W$}};
\node[below] at (3.5+17-2.4+4.9,-2.6) {\textcolor{red}{$W'$}};
\node[below] at (3.5+17-2.4+1,0.2) {\textcolor{red}{$Z_1$}};
\node[below] at (3.5+17-2.4+4,0.2) {\textcolor{red}{$Z_2$}};
\end{tikzpicture}
}
\caption{Families $B_1$ and $B_2$}\label{borsuk}
\end{figure}

The following proposition shows that if a small graph embeds into $B_1$ ($B_2$ resp.), then it can also be embedded into some graph in $\C Z_1$ ($\C Z_2$ resp.).
\begin{prop}\label{prop:embed BH*}
 Let $H$ be a graph with $\chi(H)\le 3$ and $i\in\{1,2\}$. If $H\subseteq B_i(n,k,\ell)$ with $\ell > |H|$, then $H\hookrightarrow \C Z_i$.
\end{prop}

\begin{proof}
We prove the case $i=1$ first. Let $V(B_1(n,k,\ell))=U' \cup W\cup Z_1\cup Z_2\cup Z_3$.  Suppose $B_1(n,k,\ell)$ contains a copy of $H$, say $H^*$, on vertex set $V(H^*)=U^*\cup W^*\cup Z_1^*\cup Z_2^* \cup Z_3^*$, where $U^*=V(H^*)\cap U'$, $W^*=V(H^*)\cap W$ and $Z_i^*=V(H^*)\cap Z_i$ for $i\in [3]$.
Note that $\ell>|H^*|$. We first show that $H^*[U^*\cup W^*]$ embeds into some Zykov graph $Z_k^t(A,U)$. If $\chi(H^*[U^*\cup W^*])=3$, then by \cref{LT} this graph is near-acyclic, and hence embeds into some $Z_k^t(A,U)$ by \cref{fact:zykov-universal}. If $\chi(H^*[U^*\cup W^*])\le 2$, then the same conclusion holds. Indeed, $H^*[U^*]$ is a forest, since $g(\mathrm{BH}[U'])>\ell>|H^*|$, and bipartiteness prevents any vertex of $W^*$ from having neighbours in both color classes of a component of this forest. Thus $H^*[U^*\cup W^*]$ also embeds into some $Z_k^t(A,U)$.

The remaining parts $Z_1^*, Z_2^*, Z_3^*$ can be embedded into the corresponding blowup classes of $C_5$ in the definition of $\mathcal{Z}_1$, respecting all adjacencies. Hence $H\hookrightarrow \mathcal{Z}_1$.

For $i=2$, we apply the same argument simultaneously to the two copies of $\mathrm{BH}^-$ sharing the same $U'$-part. The subgraph induced on the vertices mapped to $U'$ is a forest, and we use this forest as the common $A$-part of the two Zykov graphs. The vertices mapped to the two $W$-parts are then encoded independently in the two corresponding $U$-parts. This gives an embedding into $\mathcal Z_2$.
\end{proof}

\subsection{Extremal constructions}\label{sec:ext-constr}
Throughout this section, let $H_1,H_2$ be graphs with chromatic number 3 and $\mathrm{emb}(H_1)\le\mathrm{emb}(H_2)$. Let $k,\ell\in\mathbb{N}$ with $\ell>\max\{|H_1|,|H_2|\}$. We are now ready to construct six types of extremal graphs $G^1,\ldots,G^6$, each satisfying $\chi(G^i)\ge k$, as required in \cref{prop: lower bound}. 
In all figures, we use ``E'' to denote a $(k,\ell)$-Erd\H{o}s graph, and ``B'' to denote the subgraph $\mathrm{BH}[U']$ of a Borsuk--Hajnal graph, which has chromatic number at least $k$ and girth at least $\ell$.
Note that each $G^i$ contains either an Erd\H{o}s graph or a subgraph of a Borsuk--Hajnal graph, and hence each $G^i$ has chromatic number at least $k$.

The first three families of extremal constructions will be based on the following $2$-edge-colored complete $5$-partite graph.
We say a complete $5$-partite graph on $(V_1,V_2,V_3,V_4,V_5)$ has \emph{Ramsey coloring of $K_5$} if the edges between $V_i$ and $V_j$ are red if and only if $j-i \equiv \pm 1 \pmod{5}$.

\begin{figure}[H]
\begin{minipage}{0.3\textwidth}
\centering
\resizebox{0.97\textwidth}{!}{%
\begin{tikzpicture}
\coordinate (b1) at (2.59-1,-0.81){};
\coordinate (b2) at (4.41-1,-0.81){};
\coordinate (b3) at (4.98-1,-2.54){};
\coordinate (b4) at (3.5-1,-3.62){};
\coordinate (b5) at (2.02-1,-2.54){};
\draw [blue,line width=2pt](b1) -- (b3);
\draw [blue,line width=2pt](b1) -- (b4);
\draw [blue,line width=2pt](b2) -- (b4);
\draw [blue,line width=2pt](b2) -- (b5);
\draw [blue,line width=2pt](b3) -- (b5);

\coordinate (r1) at (3-1,0);
\coordinate (r2) at (4-1,0);
\coordinate (r3) at (5.31-1,-0.95);
\coordinate (r4) at (5.62-1,-1.9);
\coordinate (r5) at (5.12-1,-3.44);
\coordinate (r6) at (4.31-1,-4.03);
\coordinate (r7) at (2.69-1,-4.03);
\coordinate (r8) at (1.88-1,-3.44);
\coordinate (r9) at (1.38-1,-1.9);
\coordinate (r10) at (1.69-1,-0.95);
\draw [red,line width=2pt](r1) -- (r2);
\draw [red,line width=2pt](r3) -- (r4);
\draw [red,line width=2pt](r5) -- (r6);
\draw [red,line width=2pt](r7) -- (r8);
\draw [red,line width=2pt](r9) -- (r10);

\draw (2-1,0) circle (1cm);
\draw (5-1,0) circle (1cm);
\draw (1.07-1,-2.85) circle (1cm);
\draw (5.93-1,-2.85) circle (1cm);
\draw (3.5-1,-4.62) circle (1cm);

\node[below] at (3.5-1,-4.3) {\textcolor{red}{$E$}};
\node[below] at (5.93-1,-2.65) {$V_5$};
\node[below] at (4,0.2) {$V_4$};
\node[below] at (1,0.2) {$V_3$};
\node[below] at (1.07-1,-2.65) {$V_2$};
\end{tikzpicture}
}
\caption{The graph $G^1$}\label{G1}
\end{minipage}
\begin{minipage}{0.3\textwidth}
    \centering
    \resizebox{0.9 \textwidth}{!}{%
\begin{tikzpicture}
\coordinate (b1) at (2.59+8,-0.81){};
\coordinate (b2) at (4.41+8,-0.81){};
\coordinate (b3) at (4.98+8,-2.54){};
\coordinate (b4) at (3.5+8,-3.62){};
\coordinate (b5) at (2.02+8,-2.54){};
\draw [blue,line width=2pt](b1) -- (b3);
\draw [blue,line width=2pt](b1) -- (b4);
\draw [blue,line width=2pt](b2) -- (b4);
\draw [blue,line width=2pt](b2) -- (b5);
\draw [blue,line width=2pt](b3) -- (b5);

\coordinate (r1) at (3+8,0);
\coordinate (r2) at (4+8,0);
\coordinate (r3) at (5.31+8,-0.95);
\coordinate (r4) at (5.62+8,-1.9);
\coordinate (r5) at (5.12+8,-3.44);
\coordinate (r6) at (4.07+8,-4.21);
\coordinate (r7) at (2.93+8,-4.21);
\coordinate (r8) at (1.88+8,-3.44);
\coordinate (r9) at (1.38+8,-1.9);
\coordinate (r10) at (1.69+8,-0.95);
\draw [red,line width=2pt](r1) -- (r2);
\draw [red,line width=2pt](r3) -- (r4);
\draw [red,line width=2pt](r5) -- (r6);
\draw [red,line width=2pt](r7) -- (r8);
\draw [red,line width=2pt](r9) -- (r10);
\coordinate (r11) at (1.48+8,-3.77);
\coordinate (r12) at (3.29+8,-5.59);
\draw [red,line width=2pt,dashed](r11) -- (r12);
\coordinate (r13) at (7-1.48+8,-3.77);
\coordinate (r14) at (7-3.29+8,-5.59);
\draw [red,line width=2pt](r13) -- (r14);

\draw (2+8,0) circle (1cm);
\draw (5+8,0) circle (1cm);
\draw (1.07+8,-2.85) circle (1cm);
\draw (5.93+8,-2.85) circle (1cm);
\draw (3.5+8,-4.62) circle (0.7cm);
\draw (3.5+8,-5.8) circle (0.3cm);

\draw[dashed] (3.5+8,-4.9) ellipse (0.8cm and 1.3cm); % 绘制虚线的椭圆

\node[below] at (3.5+8,-5.55) {\textcolor{red}{$B$}};
\node[below] at (3.5+8,-4.42) {$V_1''$};
\node[below] at (5.93+8,-2.65) {$V_5$};
\node[below] at (13,0.2) {$V_4$};
\node[below] at (10,0.2) {$V_3$};
\node[below] at (1.07+8,-2.65) {$V_2$};
\node[below] at (12.5,-5.3) {$V_1$};
\end{tikzpicture}
}
\caption{ The graph $G^2$}\label{G2} 
\end{minipage}
\begin{minipage}{0.3\textwidth}
    \centering
    \resizebox{0.98\textwidth}{!}{%
\begin{tikzpicture}
\coordinate (b1) at (2.59+17,-0.81){};
\coordinate (b2) at (4.41+17,-0.81){};
\coordinate (b3) at (4.98+17,-2.54){};
%\coordinate (b4) at (3.5+17,-3.62){};
\coordinate (b4) at (3.5+17,-4.1){};
\coordinate (b5) at (2.02+17,-2.54){};
\draw [blue,line width=2pt](b1) -- (b3);
\draw [blue,line width=2pt](b1) -- (b4);
\draw [blue,line width=2pt](b2) -- (b4);
\draw [blue,line width=2pt](b2) -- (b5);
\draw [blue,line width=2pt](b3) -- (b5);

\coordinate (r1) at (3+17,0);
\coordinate (r2) at (4+17,0);
\coordinate (r3) at (5.31+17,-0.95);
\coordinate (r4) at (5.62+17,-1.9);
\coordinate (r5) at (5.12+17,-3.44);
\coordinate (r6) at (3.9+17,-4.3);
\coordinate (r7) at (3.1+17,-4.3);
%\coordinate (r6) at (3.95+17,-4.1);
%\coordinate (r7) at (3.05+17,-4.1);
\coordinate (r8) at (1.88+17,-3.44);
\coordinate (r9) at (1.38+17,-1.9);
\coordinate (r10) at (1.69+17,-0.95);
\draw [red,line width=2pt](r1) -- (r2);
\draw [red,line width=2pt](r3) -- (r4);
\draw [dashed,red,line width=2pt](r5) -- (r6);
\draw [dashed,red,line width=2pt](r7) -- (r8);
\draw [red,line width=2pt](r9) -- (r10);

\draw (2+17,0) circle (1cm);
\draw (5+17,0) circle (1cm);
\draw (1.07+17,-2.85) circle (1cm);
\draw (5.93+17,-2.85) circle (1cm);
\draw (3.5+17,-4.6) circle (0.5cm);

\node[below] at (3.5+17,-4.3) {\textcolor{red}{$B$}};
\node[below] at (5.93+17,-2.65) {$V_5$};
\node[below] at (22,0.2) {$V_4$};
\node[below] at (19,0.2) {$V_3$};
\node[below] at (1.07+17,-2.65) {$V_2$};
\end{tikzpicture}
}
\caption{The graph $G^3$}\label{G3}
\end{minipage}
\end{figure}

%---------------------G1--------------------------------
\begin{cons}
 Let $G^1$ be an $n$-vertex graph on vertex set $\bigcup_{i\in[5]} V_i$ obtained as follows.
\begin{itemize}
\item Start with a complete $5$-partite graph on $(V_1,V_2,V_3,V_4,V_5)$ with Ramsey coloring of $K_5$ and $|V_i|=n/5$ for $i\in[5]$. 
\item Embed an $n/5$-vertex $(k,\ell)$-Erd\H{o}s graph into $V_1$ and color all its edges red. 
\end{itemize}
Then $\delta(G^1)=4n/5$ and $\chi(G^1)\ge k$.
\end{cons}

\begin{property}[$G^1$]
  If $H_1   \not\hookrightarrow\C Z_0$ and $H_2 \not\hookrightarrow \C C_5$, then $G_r^1$ is $H_1$-free and $G_b^1$ is $H_2$-free.
\end{property}  

\begin{proof}
  Since $G_b^1\in\C C_5$ and $H_2 \not\hookrightarrow \C C_5$, it follows that $G_b^1$ is $H_2$-free. If $H_1$ can be embedded into $G_r^1$, then let $W_i=V(H_1)\cap V_i$ for $i\in[5]$.\footnote{With a slight abuse of notation, we also denote by $H_1$ a copy of $H_1$ in $G_r^1$.}
  By the definition of $G^1$, for each $2\le i\le 5$, $W_i$ is an independent set.
  Moreover, $H_1[W_1]$ is a forest, since $G^1[V_1]$ has girth greater than $|H_1|$.
  It follows that $H_1 \hookrightarrow \C Z_0$, a contradiction.
\end{proof}

%--------------------------G2------------------------
\begin{cons}
 Let $G^2$ be an $n$-vertex graph on vertex set $\bigcup_{i\in[5]} V_i$ obtained as follows.
\begin{itemize}
\item Start with a complete $5$-partite graph on $(V_1,V_2,V_3,V_4,V_5)$ with Ramsey coloring of $K_5$, $|V_1|=n/9$ and $|V_2|=|V_3|=|V_4|=|V_5| = 2n/9$. 

\item Split $V_1=V'_1 \cup V''_1$ with $|V'_1| = o(n)$ and replace the bipartite graph between $V_1'$ and $V_2$ by a copy of $\mathrm{BH}^-$
so that $(V_1',V_2,V_3,V_4,V_5)$ induces a red $B_1((8/9+o(1))n,k,\ell)$.
\end{itemize}
Then $\delta(G^2)=(7/9-o(1))n$ and $\chi(G^2)\ge k$.
\end{cons}

\begin{property}[$G^2$]
  If $H_1\not\hookrightarrow \C Z_1$ and $H_2\not\hookrightarrow \C C_5$, then $G_r^2$ is $H_1$-free and $G_b^2$ is $H_2$-free.
\end{property}

\begin{proof}
  Since $G_b^2\in\C C_5$ and $H_2\not\hookrightarrow \C C_5$, $G_b^2$ is $H_2$-free. Suppose that $G_r^2$ contains a copy of $H_1$. Let $W_1''=V(H_1)\cap V_1''$.
  Let $H^{*}=H_1[V(H_1)\setminus W_1'']$. Then $H^{*}\subseteq B_1(n,k,\ell)$.
  By \cref{prop:embed BH*}, $H^{*}\hookrightarrow \C Z_1$.
  Since $W_1''$ is an independent set and all its edges to the rest of the graph follow the $C_5$-blowup pattern, by \cref{fact:embde C5} we obtain $H_1 \hookrightarrow \C Z_1$, a contradiction.
\end{proof}

\begin{cons}
 Let $G^3$ be an $n$-vertex graph on vertex set $\bigcup_{i\in[5]} V_i$ obtained as follows.
\begin{itemize}
\item Start with a complete $5$-partite graph on $(V_1,V_2,V_3,V_4,V_5)$ with Ramsey coloring of $K_5$, $|V_1|=o(n)$ and $|V_2|=|V_3|=|V_4|=|V_5| = n/4-o(n)$.

\item Replace the bipartite graphs between $V_1$ and $V_2$, and between $V_1$ and $V_5$, by copies of $\mathrm{BH}^-$
so that $(V_1,V_2,V_3,V_4,V_5)$ induces a red $B_2(n,k,\ell)$.
\end{itemize}
Then $\delta(G^3)=\left(3/4-o(1)\right)n$ and $\chi(G^3)\ge k$.
\end{cons}

\begin{property}[$G^3$]
  If $H_1\not\hookrightarrow \C Z_2$ and $H_2\not\hookrightarrow \C C_5$, then $G_r^3$ is $H_1$-free and $G_b^3$ is $H_2$-free.
\end{property}

\begin{proof}
  Since $G_b^3\in\C C_5$ and $H_2\not\hookrightarrow \C C_5$, $G_b^3$ is $H_2$-free. If $H_1$ embeds into $G_r^3$, then $H_1\subseteq B_2(n,k,\ell)$. By \cref{prop:embed BH*}, this implies $H_1 \hookrightarrow \C Z_2$, a contradiction.
\end{proof}

For the last three constructions $G^4,G^5,G^6$, we do not impose conditions on the embeddabilities of $H_1$ and $H_2$, and instead only use $\delta_{\chi}(H_i)$.
Without loss of generality, we assume that $\delta_{\chi}(H_1)\ge \delta_{\chi}(H_2)$.

\begin{figure}[!ht]
\begin{minipage}{0.3\textwidth}
    \centering
    \resizebox{0.9\textwidth}{!}{%
\begin{tikzpicture}
\coordinate (r1) at (2,5);
\coordinate (r2) at (4,5);
\coordinate (r3) at (2,1);
\coordinate (r4) at (4,1);
\draw [red,line width=2pt](r1) -- (r2);
\draw [red,line width=2pt](r3) -- (r4);

\coordinate (b1) at (1,4);
\coordinate (b2) at (1,2);
\coordinate (b3) at (5,4);
\coordinate (b4) at (5,2);
\coordinate (b5) at (1.71,4.29);
\coordinate (b6) at (4.29,1.71);
\coordinate (b7) at (4.29,4.29);
\coordinate (b8) at (1.71,1.71);
\draw [blue,line width=2pt](b1) -- (b2);
\draw [blue,line width=2pt](b3) -- (b4);
\draw [blue,line width=2pt](b5) -- (b6);
\draw [blue,line width=2pt](b7) -- (b8);

\draw (1,1) circle (1cm);
\draw (5,1) circle (1cm);
\draw (1,5) circle (1cm);
\draw (5,5) circle (1cm);
\node[below] at (1,1.3) {\textcolor{red}{$E$}};
\node[below] at (5,1.2) {$V_2$};
\node[below] at (5,5.2) {$V_3$};
\node[below] at (1,5.2) {$V_4$};
\end{tikzpicture}
}
\caption{The graph $G^4$}\label{G4}
\end{minipage}
\begin{minipage}{0.3\textwidth}
   \centering
    \resizebox{0.76\textwidth}{!}{%
\begin{tikzpicture}
\coordinate (r1) at (3.5+8,5.2);
\coordinate (r2) at (3.5+8,2);
\coordinate (r3) at (2+8,3.5);
\coordinate (r4) at (5+8,3.5);
\coordinate (r5) at (3.5+8,0);
\coordinate (r6) at (3.5+8,-0.5);
\draw [red,line width=2pt](r1) -- (r2);
\draw [red,line width=2pt](r3) -- (r4);
\draw [red,line width=2pt,dashed](r5) -- (r6);

\coordinate (b1) at (2.93+8,5.43);
\coordinate (b2) at (1.71+8,4.21);
\coordinate (b3) at (1.71+8,2.79);
\coordinate (b4) at (2.79+8,1.71);
\coordinate (b5) at (1+8,2.5);
\coordinate (b6) at (3.21+8,-0.59);
\coordinate (b1') at (7-2.93+8,5.43);
\coordinate (b2') at (7-1.71+8,4.21);
\coordinate (b3') at (7-1.71+8,2.79);
\coordinate (b4') at (7-2.79+8,1.71);
\coordinate (b5') at (7-1+8,2.5);
\coordinate (b6') at (7-3.21+8,-0.59);
\draw [blue,line width=2pt](b1) -- (b2);
\draw [blue,line width=2pt](b3) -- (b4);
\draw [blue,line width=2pt](b5) -- (b6);
\draw [blue,line width=2pt](b1') -- (b2');
\draw [blue,line width=2pt](b3') -- (b4');
\draw [blue,line width=2pt](b5') -- (b6');

\draw (1+8,3.5) circle (1cm);
\draw (3.5+8,1) circle (1cm);
\draw (6+8,3.5) circle (1cm);
\draw (3.5+8,6) circle (0.8cm);
\draw (3.5+8,-1) circle (0.5cm);
\node[below] at (3.5+8,-0.7) {\textcolor{red}{$B$}};
\node[below] at (3.5+8,1.2) {$V_2$};
\node[below] at (3.5+8,6.3) {$V_3$};
\node[below] at (9,3.7) {$V_4$};
\node[below] at (14,3.7) {$V_5$};
\end{tikzpicture}
}
\caption{The graph $G^5$}\label{G5}
\end{minipage}
\begin{minipage}{0.3\textwidth}
    \centering
    \resizebox{0.97\textwidth}{!}{%
\begin{tikzpicture}
\coordinate (r1) at (2+17,2+3);
\coordinate (r2) at (4.5+17,2+3);
\draw [red,line width=2pt](r1) -- (r2);

\coordinate (b1) at (1.5+17,1.11+3);
\coordinate (b2) at (2.75+17,-1.03+3);
\coordinate (b3) at (6.5-1.5+17,1.11+3);
\coordinate (b4) at (6.5-2.75+17,-1.03+3);
\draw [blue,line width=2pt](b1) -- (b2);
\draw [blue,line width=2pt](b3) -- (b4);

\draw (1+17,2+3) circle (1cm);
\draw (5.5+17,2+3) circle (1cm);
\draw (3.25+17,-1.9+3) circle (1cm);
\node[below] at (3.25+17,-1.6+3) {\textcolor{red}{$E$}};
\end{tikzpicture}
}
\caption{The graph $G^6$}\label{G6}
\end{minipage}
\end{figure}

%-----------------------------G4-----------------------------

\begin{cons}
 Let $G^4$ be an $n$-vertex graph on vertex set $\bigcup_{i\in[4]}V_i$ obtained as follows.
\begin{itemize}
\item Start with a complete $4$-partite graph on $(V_1,V_2,V_3,V_4)$, where the edges between $V_1$ and $V_2$, and between $V_3$ and $V_4$, are colored red, and all other edges blue, with $|V_i|=n/4$ for $i\in[4]$.

    \item Embed an $n/4$-vertex $(k,\ell)$-Erd\H{o}s graph into $V_1$ and color all its edges red. 
\end{itemize}
Then $\delta(G^4)=3n/4$ and $\chi(G^4)\ge k$.
\end{cons}

\begin{property}[$G^4$]
  If $\delta_{\chi}(H_1)=1/2$, then $G_r^4$ is $H_1$-free and $G_b^4$ is $H_2$-free.
\end{property}

\begin{proof}
Since $G_b^4$ is bipartite, it is $H_2$-free.
Suppose, for a contradiction, that $G^4_r$ contains a copy of $H_1$.
Let $A_i=V(H_1)\cap V_i$ for $i\in[4]$. Since there are no red edges
between $V_1\cup V_2$ and $V_3\cup V_4$, the sets $A_1\cup A_2$ and
$A_3\cup A_4$ are unions of components of this copy. Moreover,
$H_1[A_1]$ is a forest, since $G^4_r[V_1]$ has girth larger than
$|H_1|$, and $H_1[A_3\cup A_4]$ is bipartite.
Then deleting an independent set from $H_1$ leaves a forest. 
Hence $\mathcal M(H_1)$ contains a forest, contradicting
\cref{thm:characterization}, since $\delta_\chi(H_1)=1/2$.
\end{proof}

%----------------G5------------------

\begin{cons}
   Let $G^5$ be an $n$-vertex graph on vertex set $\bigcup_{i\in[5]}V_i$ obtained as follows.
\begin{itemize}
\item Start with a complete bipartite graph on partite sets $V_1\cup V_2\cup V_3$ and $V_4\cup V_5$ with blue edges and $|V_1|= o(n), |V_3| =n/7-o(n), |V_2|=|V_4|=|V_5|=2n/7.$

\item The red graph consists of a copy of $\mathrm{BH}(m,k,\ell)$ on $(V_1,V_2,V_3)$,
where $m=|V_1|+|V_2|+|V_3|=(3/7+o(1))n$, together with a complete bipartite graph on $(V_4,V_5)$.
\end{itemize}
Then $\delta(G^5)=(5/7-o(1))n$ and $\chi(G^5)\ge k$.
\end{cons}

\begin{property}[$G^5$]
  If $\delta_{\chi}(H_1)=1/3$, then $G_r^5$ is $H_1$-free and $G_b^5$ is $H_2$-free.
\end{property}
  
\begin{proof}
Since $G_b^5$ is bipartite, it is $H_2$-free. 
Suppose, for a contradiction, that $G^5_r$ contains a copy of $H_1$.
Let $A=V(H_1)\cap (V_1\cup V_2\cup V_3)$, and $B=V(H_1)\cap (V_4\cup V_5)$.
As $H_1[A]\subseteq \mathrm{BH}(m,k,\ell)$ and $|H_1[A]|<\ell$, \cref{LT}
implies that $H_1[A]$ is near-acyclic. Since $H_1[B]$ is bipartite,
this implies that $H_1$ is near-acyclic, contradicting
\cref{thm:characterization}, since $\delta_\chi(H_1)=1/3$.
\end{proof}

%-------------------------G6-------------------------

\begin{cons}
 Let $G^6$ be an $n$-vertex graph on vertex set $\bigcup_{i\in[3]}V_i$ obtained as follows.
\begin{itemize}
\item Start with a complete $3$-partite graph on $(V_1,V_2,V_3)$ with $|V_i|=n/3$, $i\in[3]$, red edges in $[V_2,V_3]$ and blue edges elsewhere.
    \item Embed an $n/3$-vertex $(k,\ell)$-Erd\H{o}s graph into $V_1$ and color all its edges red. 
\end{itemize}
Then $\delta(G^6)=2n/3$ and $\chi(G^6)\ge k$.
\end{cons}

We can easily derive the following property of $G^6$.
 \begin{property}[$G^6$]
  For any $H$ with $\chi(H)=3$ and $|H|<\ell$, both $G_r^6$ and $G_b^6$ are $H$-free.
\end{property}

\section{Core Lemmas}\label{core}\label{sec: core}

In this section, we will prove the following lemma, which plays a key role in the proof of \cref{thm:main}. 

\begin{lemma}\label{lmm:embed key}
Let $1>d \gg \delta \gg \varepsilon \gg \frac{1}{M}>0$ and $H$ be a graph with $\chi(H)=3$. Then there exists $C=C(H,d,\delta,\varepsilon,M)$ such that the following hold for sufficiently large $n$. Suppose $V_1,V_2,V_3, V_4$ are four disjoint subsets of $V(G)$ with $|V_1|=|V_2|=|V_3|=|V_4| =n$ and $(V_i,V_{i+1})$ is $(\varepsilon,d)$-regular for $i=1,2,3$. Then for any $X\subseteq V(G)$ with $\chi(G[X])>C$ and $|X|\le Mn$, we have the following.
\begin{enumerate}
    \item[\rm (i)]  If $H\hookrightarrow \C Z_0$, and $(X,V_1)$ and $(X,V_4)$ are both $(1/2+\delta)$-dense, then $H \subseteq G$.

    \item[\rm (ii)]  If $H\hookrightarrow \C Z_1$, and $(X,V_1)$ is $\delta$-dense and $(X,V_4)$ is $(1/2+\delta)$-dense, then $H \subseteq G$.

    \item[\rm (iii)]  If $H\hookrightarrow \C Z_2$, and $(X,V_1)$ and $(X,V_4)$ are both $\delta$-dense, then $H \subseteq G$.

%\item[\rm (iv)] If $|X| \le  C'\cdot n$ for some constant $C'$, we can not require $X$ is disjoint from $V_1,V_2,V_3,V_4$.
\end{enumerate}
\end{lemma}  

We also record two corollaries of the following simple result, which will be useful later.       

\begin{prop}\label{prop:intersection}
   Let $ \alpha \in (0,1)$ and $V_1, V_2, \cdots, V_m$ be subsets of $[n]$ each with size at least $\alpha n$. 
   If $m\ge 2/\alpha$, then there exist $1\le i<j\le m$ such that $|V_i\cap V_j|\ge \frac{\alpha^2}{4}n.$
\end{prop}

\begin{proof}
   Suppose to the contrary that for any  $1\le i<j\le m$, we have $|V_i\cap V_j| < \alpha^2n/4$.
    By the inclusion-exclusion principle, we have 
\[n \ge \left| \bigcup_{i=1}^{\ceil{2/\alpha}} V_i \right|\ge \sum_{i=1}^{\ceil{2/\alpha}} |V_i| - \sum_{1\le i< j \le \ceil{2/\alpha}} |V_i\cap V_j| > \frac{2}{\alpha} \alpha n - \binom{\ceil{2/\alpha}}{2} \frac{\alpha^2}{4}n >n,
\]
a contradiction. 
\end{proof}

Iterating the above proposition yields the following two corollaries.
\begin{cor}\label{lmm:in-ex}
    Let $t\in \mathbb{N}$ and $\varepsilon\in (0,1)$. Then there exist $\alpha =\alpha(\varepsilon,t)$ and $m=m(\varepsilon,t)$ such that for any $V_1, V_2, \cdots, V_m\subseteq [n]$ each with size at least $\varepsilon n$, there exist $1\le i_1< i_2<\dots < i_t\le m$ with $\left|\bigcap_{j=1}^{t} V_{i_j}\right| \ge \alpha n$.
\end{cor}

%The following result is a straightforward corollary.
\begin{cor}\label{coro: large common}
  Let $t\in \mathbb{N}$ and $\varepsilon\in (0,1)$, and let $n$ be a sufficiently large integer. Then there exist $\beta = \beta(\varepsilon,t)$ and $C=C(\varepsilon,t)$ such that the following holds.
Let $X$, $V_1$ and $V_2$ be three disjoint subsets of vertices of an $n$-vertex graph $G$ with $|X|\ge C$.
If $(X,V_1)$ and $(X,V_2)$ are both $\varepsilon$-dense, then there exists a subset $W\subseteq X$ with $|W|=t$ and $|N(W)\cap V_1|\ge \beta|V_1|$ and $|N(W)\cap V_2|\ge \beta|V_2|$.
\end{cor}

\subsection{Proof of Lemma~\ref{lmm:embed key}}

We shall derive~\cref{lmm:embed key} from the following two lemmas, \cref{lmm:re-par} and \cref{lmm:xiaobaima}. To state \cref{lmm:re-par}, we need to introduce some definitions.
\begin{definition}[$t$-goodness]
 Let $t$ be an integer and $G$ be a graph.
 Suppose $V_1,V_2,V_3,V_4$ are four disjoint subsets of $V(G)$.
 For $S\subseteq V_1$ and $T\subseteq V_4$,  we say a pair  $(S,T)$ is \emph{$t$-good} with respect to $(V_2,V_3)$ if there exist subsets $W_2\subseteq V_2$ and $W_3\subseteq V_3$ with $|W_2|=|W_3|=t$ such that all of $G[S,W_2]$, $G[W_2,W_3]$, and $G[W_3,T]$ are complete bipartite graphs.
\end{definition}

Let $\alpha,\varepsilon>0$ be two real numbers.
Given $X\subseteq V(G)$, we say a partition $X=S_0\cup S_1\cup S_2\cup\cdots\cup S_p$ is an \emph{$(\alpha,\varepsilon)$-partition} (of $X$ in $G$) if $|S_1| =|S_2| =\cdots = |S_p| = \alpha |X|$, and $|S_0| < \varepsilon |X|$.

\begin{lemma}\label{lmm:re-par}
 Let $\varepsilon,d$ be two real numbers and $n,t$ be two integers with $1>d,1/t\gg \varepsilon\gg 1/n>0$. Let $G$ be a graph.
 Suppose that $V_1, V_2, V_3,V_4$ are four disjoint subsets of $V(G)$ with $|V_1|=|V_2|=|V_3|=|V_4|=n$ such that $(V_i, V_{i+1})$ is $(\varepsilon, d)$-regular for $i = 1, 2, 3$.
Then, there exists a real number $\alpha:=\alpha(\varepsilon,d,t)<\varepsilon$ and an $(\alpha, \varepsilon)$-partition $V_1=S_0 \cup S_1 \cup S_2 \cup \cdots \cup S_p$ such that the following holds.
\begin{itemize} 
    \item For each $i \in [p]$, there is an $(\alpha, \varepsilon)$-partition $V_4 = T_{0}^i \cup T_{1}^i \cup T_{2}^i \cup \cdots \cup T_{p}^i$, such that $(S_i, T_{j}^i)$ is $t$-good with respect to $(V_2,V_3)$ for any $j \in [p]$.
\end{itemize}  
\end{lemma}

\begin{lemma}\label{lmm:xiaobaima}
    Let $k$ and $t$ be two positive integers.
    For any $\delta >0$, there exist a constant $C$ and a sufficiently large integer $m$ such that
    for any  graph $G$, and three disjoint subsets $W,V_1,V_2$ of $V(G)$  with $\chi(G[W]) \ge C$ and $|W|\le |V_1|=|V_2|=m$, we have the following.
\begin{itemize}
    \item[\rm{(i)}] If $(W,V_1)$ and $(W,V_2)$ are both $\delta$-dense,
    then $G$ contains two copies of $Z^t_k$ sharing the same $A$-part, say $Z^t_k(A,U_1),Z^t_k(A,U_2)$, with $(A,U_1,U_2)\subseteq(W,V_1,V_2)$.

    \item[\rm{(ii)}] If $(W,V_1)$ and $(W,V_2)$ are both $(1/2+\delta)$-dense, 
    then there exists a copy of $\C F^t_k$ in $G[W]$, say $T$, such that $|N(T) \cap V_1|\ge t$ and $|N(T) \cap V_2|\ge t$.

    \item[\rm{(iii)}] If $(W,V_1)$ is $\delta$-dense and $(W,V_2)$ is $(1/2+\delta)$-dense,
     then there exists a copy of $Z^t_k(A,U)$ with $(A,U) \subseteq(W,V_1)$ and a vertex set $S \subseteq V_2$ with $|S| \ge t$ such that $G[A,S]$ is  complete bipartite.
\end{itemize}
\end{lemma}

\begin{proof}[Proof of \cref{lmm:embed key}]
   Fix integers $q\gg t_0\gg |H|$, depending only on $H$. We first demonstrate $\rm{(i)}$. By \cref{lmm:re-par}, there exists a real number $\alpha:=\alpha(\varepsilon,d,t_0)< \varepsilon$ and an $(\alpha, \varepsilon)$-partition $V_1=S_0 \cup S_1 \cup S_2 \cup \cdots \cup S_p$ such that for each $S_i$, there is an $(\alpha, \varepsilon)$-partition $V_4 = T_{0}^i \cup T_{1}^i \cup T_{2}^i \cup \cdots \cup T_{p}^i$ satisfying that $(S_i, T_{j}^i)$ is $t_0$-good with respect to $(V_2,V_3)$.

Let $V_1'=V_1\setminus S_0$ and $V_4^i=V_4\setminus T_{0}^i$.
Then for any $x\in X$, $|N(x)\cap V_1'|\ge (1/2+\delta-\varepsilon)|V_1'|$ and $|N(x)\cap V_4^i|\ge (1/2+\delta-\varepsilon)|V_4^i|$.   
For each $i,j\in[p]$, let 
$$X_{i,j}=\left\{x\in X\setminus(S_i\cup T^i_j): d_{S_i}(x)\ge (1/2+\delta -\varepsilon) |S_i|, \text{ and }  d_{T^{i}_j}(x)\ge (1/2+\delta -\varepsilon) |T^i_j|\right\}.$$
Then by averaging and $\alpha <\varepsilon \ll \delta$, we have $\bigcup_{i,j} X_{i,j} = X$, which implies that 
$$\chi(G[X]) \le \sum_{i,j}\chi (G[X_{i,j}]).$$ 
As $\chi(G[X])\ge C$ is sufficiently large, averaging again we see that there exist $i,j$ with  $\chi (G[X_{i,j}]) \ge C/p^2 >C'M\frac{1}{\alpha}$, where $C'$ is the constant obtained from \cref{lmm:xiaobaima} with $(k,t,\delta) = (|H|,q, \delta-\varepsilon).$
Since $|X|< Mn$, there exists $X'_{i,j} \subseteq X_{i,j}$ with $|X'_{i,j}| \le \alpha n$ and $\chi (G[X'_{i,j}]) \ge C'$.
Then by the definition of $X'_{i,j}$, we can apply   \cref{lmm:xiaobaima}$\rm{(ii)}$ with $(W,V_1,V_2)=(X'_{i,j},S_i,T_j^i)$ to get a copy of $\C F_{|H|}^{q}$, say $T$, in $G[X'_{i,j}]$ such that $|N(T)\cap S_i|,~|N(T)\cap T_j^i|\ge q$. Recall that $(S_i,T_{j}^i)$ is $t_0$-good with respect to $(V_2,V_3)$ and $H\hookrightarrow\C Z_0$. Let $W_2\subseteq V_2$ and $W_3\subseteq V_3$ be the sets witnessing this $t_0$-goodness.

Since $q\gg t_0\gg |H|$, after discarding the vertices of $T$ that lie in $W_2\cup W_3$, the copy of $\C F_{|H|}^{q}$ still contains enough unused copies of every tree on at most $|H|$ vertices. Similarly, the sets $N(T)\cap S_i$, $N(T)\cap T_j^i$, $W_2$ and $W_3$ contain enough unused vertices for the embedding. Thus, using these unused vertices, we can embed $H$ into $G$ as desired (see~\cref{fig:41}).

\begin{figure}[!ht]
    \centering
    \resizebox{0.35\textwidth}{!}{%
\begin{tikzpicture}
\coordinate (r1'') at (-1,0);
\coordinate (r2'') at (2,0);
\coordinate (r3'') at (2,0);
\coordinate (r4'') at (2.93,-2.85);
\coordinate (r5'') at (2.93,-2.85);
\coordinate (r6'') at (0.5,-4.9);
\coordinate (r7'') at (0.5,-4.9);
\coordinate (r8'') at (1.07-3,-2.85);
\coordinate (r9'') at (1.07-3,-2.85);
\coordinate (r10'') at (-1,0);
\draw [red,line width=2pt](r1'') -- (r2'');
\draw [red,line width=2pt](r3'') -- (r4'');
\draw [red,line width=2pt](r5'') -- (r6'');
\draw [red,line width=2pt](r7'') -- (r8'');
\draw [red,line width=2pt](r9'') -- (r10'');

\draw (-2,-5.5) rectangle (3,-4.3);

\draw (-1,0) circle (1cm);
\filldraw[fill=white] (-1,0) circle (0.5cm);

\draw (2,0) circle (1cm);
\filldraw[fill=white] (2,0) circle (0.5cm);

\draw (1.07-3,-2.85) circle (1cm);
\filldraw[fill=white] (1.07-3,-2.85) circle (0.5cm);

\draw (5.93-3,-2.85) circle (1cm);
\filldraw[fill=white] (2.93,-2.85) circle (0.5cm);

\filldraw[fill=white] (0.5,-4.9) circle (0.5cm);

\node[below] at (0.5,-4.55) {\textcolor{red}{$\mathcal{F}_{|H|}^{q}$}};
\node[below] at (8.1-10,-2.55) {\textcolor{red}{$S_i$}};
\node[below] at (9-10,0.25) {\textcolor{red}{$W_2$}};
\node[below] at (12-10,0.25) {\textcolor{red}{$W_3$}};
\node[below] at (2.95,-2.5) {\textcolor{red}{$T_j^i$}};
\node at (-1.5,-5) {$X_{i,j}$};
\node at (-3.2,-3.3) {$V_1$};
\node at (4.25,-3.3) {$V_4$};
\node at (-2.2,0.55) {$V_2$};
\node at (3.2,0.55) {$V_3$};

\end{tikzpicture}
}
\caption{Embed $H$ into $G$}\label{fig:41}
\end{figure}
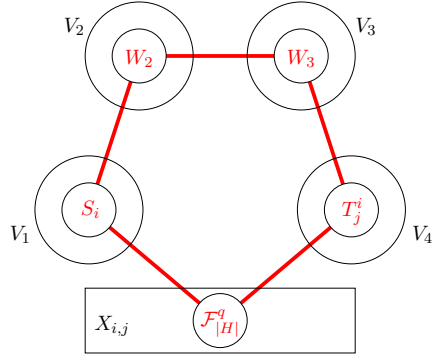

The proofs of (ii) and (iii) follow analogously, with the same choice of the auxiliary constants $q\gg t_0\gg |H|$, invoking \cref{lmm:xiaobaima}(iii) and (i), respectively.
\end{proof}

\subsection{Proof of Lemma~\ref{lmm:re-par}}
  To prove \cref{lmm:re-par}, it suffices to show that there exists $\alpha:=\alpha(\varepsilon,d,t)$ with the following property: 
\begin{itemize}
\item 
   for any $X\subseteq V_1$ with $|X|\ge \varepsilon |V_1|$, there exist $S\subseteq X$ of size $|S|= \alpha |V_1|$ and an $(\alpha,\varepsilon)$-partition  $V_4 = T_{0} \cup T_{1} \cup T_{2} \cup \cdots \cup T_{p}$, such that for every $i\in [p]$, $(S, T_{i})$ is $t$-good with respect to $(V_2,V_3)$.
  \end{itemize}    
Indeed, given the above property, we can successively select sets $S_1,\cdots,S_p$ that meet the requirements of \cref{lmm:re-par} until $|V_1\setminus(S_1\cup \cdots\cup S_p)|\le  \varepsilon |V_1|$. 
   Let $S_0=V_1\setminus(S_1\cup \cdots\cup S_p)$, then $V_1=S_0 \cup S_1 \cup S_2 \cup \cdots \cup S_p$ is the desired $(\alpha, \varepsilon)$-partition.
\begin{figure}[H]
    \centering
\begin{tikzpicture}[scale=0.8]
  % 画第一个椭圆
  \draw[black,line width=0.901pt] (0,0) ellipse (1 and 2);
  \draw[red,line width=0.901pt] (0,0) ellipse (0.5 and 1);
  \draw[red,line width=0.901pt] (0,1) -- (3,0.4);
  \draw[red,line width=0.901pt] (0,-1) -- (3,-0.4);
  \node at (0,0) {$S$};
  \node at (0,-1.5) {$X$};
  % 画第二个椭圆
  \draw[black,line width=0.901pt] (3,0) ellipse (1 and 2);
  \draw[red,line width=0.901pt] (3,0) ellipse (0.4 and 0.4);
  \draw[red,line width=0.901pt] (3,0.4) -- (6,0.8);
  \draw[red,line width=0.901pt] (3,-0.4) -- (6,-0.8);
  \node at (3,0) {$W_2$};
  \node at (3,-1.5) {$V_2$};
  % 画第三个椭圆
  \draw[black,line width=0.901pt] (6,0) ellipse (1 and 2);
  \draw[dashed, red,line width=0.901pt] (6,0) ellipse (0.6 and 0.8);
  \draw[red,line width=0.901pt] (6,0) ellipse (0.4 and 0.4);
  \draw[red,line width=0.901pt] (6,0.4) -- (9,1);
  \draw[red,line width=0.901pt] (6,-0.4) -- (9,-1);
  \node at (6,0) {$W_3$};
  \node at (6,-1.5) {$V_3$};
  % 画第四个椭圆
  \draw[black,line width=0.901pt] (9,0) ellipse (1 and 2);
  \draw[red,line width=0.901pt] (9,0) ellipse (0.5 and 1);  
  \node at (9,0) {$T_i$};
  \node at (9,-1.5) {$V_4$};
\end{tikzpicture}
\caption{An illustration of the proof of \cref{lmm:re-par}}
\end{figure}
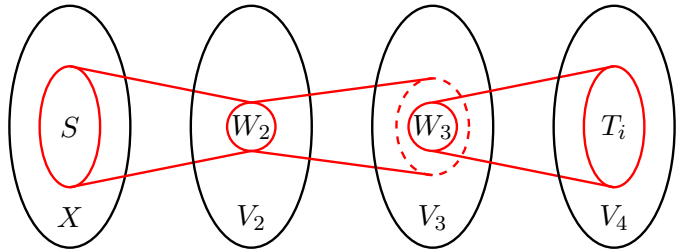
\begin{claim}\label{cl:select-ST}
   For any $X\subseteq V_1$ with $|X|\ge \varepsilon |V_1|$, there exists $W_2\subseteq V_2$ with $|W_2|=t$ such that $|N(W_2)\cap X|\ge \beta |X|$ and $|N(W_2)\cap V_3|\ge \beta |V_3|$ for some $\beta=\beta(d,t)\gg \varepsilon$.
\end{claim}

\begin{poc}
   Let $Z:=\{x\in V_2: d_X(x) < (d-2\varepsilon)|X|\}$, $Z':=\{x\in V_2:  d_{V_3}(x) < (d-2\varepsilon)|V_3|\}$ and $$V_2'=V_2\setminus(Z\cup Z')=\{x\in V_2: d_X(x) \ge (d-2\varepsilon)|X| \text{ and }   d_{V_3}(x) \ge (d-2\varepsilon)|V_3|\}.$$ As $|X| \ge \varepsilon |V_1|$, if $|Z|\ge \varepsilon |V_2|$, then due to the $(\varepsilon,d)$-regularity of $(V_1,V_2)$, $e(X,Z)\ge (d-\varepsilon)|X||Z|$, contradicting the choice of $Z$. Thus, $|Z|< \varepsilon |V_2|$. Similarly, we have $|Z'|< \varepsilon |V_2|$ and therefore $|V_2'|=|V_2|-|Z\cup Z'|\ge (1-2\varepsilon)|V_2|$.
   Furthermore, we have $(V'_2,X)$ and $(V'_2,V_3)$ are $\frac{d}{2}$-dense, as $d\gg \varepsilon$. The desired $W_2\subseteq V_2$ can be obtained by invoking \cref{coro: large common} with $(\varepsilon,t)=(d/2,t)$.
\end{poc}

\begin{claim}\label{cl:select-T}
    For any $W\subseteq V_3$ with $|W|\ge 2\varepsilon |V_3|$ and $Y\subseteq V_4$ with $|Y|\ge \varepsilon |V_4|$, there exists $W_3\subseteq W$ with $|W_3|=t$ such that $|N(W_3)\cap Y|\ge \gamma|Y|$
    for some $\gamma=\gamma(d,t) \gg \varepsilon$.
\end{claim}

\begin{poc}
   Define $W': = \{w\in W: d_Y(w) \ge (d-2\varepsilon)|Y|\}$ and $W'':=W\setminus W'.$ Again by the choice of $W''$ and $(V_3,V_4)$ being a regular pair, $|W''| < \varepsilon |V_3|$.
  Thus we have $|W'| \ge \varepsilon |V_3|$ and $(W',Y)$ is $\frac{d}{2}$-dense.
  Then the desired $W_3\subseteq W$ can be obtained by \cref{lmm:in-ex} with $(\varepsilon,t)=(d/2,t)$.
\end{poc}
Take $\alpha \ll \varepsilon,\gamma,\beta$.
Let $X\subseteq V_1$ with $|X|\ge \varepsilon |V_1|$. 
Using \cref{cl:select-ST}, we can find a subset $S\subseteq X$ with $|S|=\alpha|V_1| \le \beta |X|$ and $W_2\subseteq V_2$ with $|W_2|=t$, such that $G(W_2,S)$ is a complete bipartite graph and $|N(W_2)\cap V_3|\ge \beta|V_3|$. Then by \cref{cl:select-T} with~$W=N(W_2)\cap V_3$, for any $Y \subseteq V_4$ with $|Y| \ge \varepsilon|V_4|$, there exists $W_3 \subseteq N(W_2)\cap V_3$ with $|W_3|=t$  and $T\subseteq Y$ with $|T| =\alpha |V_4| \le \gamma |Y|$ such that $G[W_3,T]$ is a complete bipartite graph, which implies that $(S,T)$ is  $t$-good with respect to $(V_2,V_3)$.
So we can successively select sets $T_1,T_2, \dots T_p$ such that $(S,T_i)$ is $t$-good with respect to $(V_2,V_3)$ until $|V_4 \setminus \{T_1\cup \dots\cup T_p \}| \le \varepsilon|V_4|$. This finishes the proof of \cref{lmm:re-par}.

\subsection{Proof of Lemma~\ref{lmm:xiaobaima}}
A random variable $X$ is said to follow the hypergeometric distribution, denoted as $X\sim \mathrm{Hyp}(N,K,n)$, if its probability mass function is given by
$$\mathbb{P}(X=k) = \frac{\binom{K}{k}\binom{N-K}{n-k}}{\binom{N}{n}},$$
where  $N$ is the population size, $K$ is the number of success states in the population, $n$ is the number of draws (i.e. quantity drawn in each trial), and $k$ is the number of observed successes.

\begin{lemma}[Tail bounds, \cite{409cf137-dbb5-3eb1-8cfe-0743c3dc925f}]\label{lmm:tail}
    Let $X \sim \mathrm{Hyp}(N,K,n)$ and $p= K/N$. Then for $0<t<p$, 
    $$\mathbb{P}(X\le (p-t)n) \le e^{-2t^2n}.$$
\end{lemma}

\begin{proof}[Proof of \cref{lmm:xiaobaima}.]
  Assume that $V_1=\{x_1,x_2,\cdots, x_m\}$ and $V_2=\{y_1,y_2,\cdots,y_m\}$. Consider a uniform random permutation $\sigma$ of $[m]$.
  We construct an auxiliary graph $\Gamma_\sigma$ as follows.
\begin{itemize}
    \item The vertex set $V(\Gamma_\sigma)$ is $W\cup Z$, where $Z=\{z_1,\ldots,z_m\}$.
    \item The edge set $E(\Gamma_\sigma)$ is $E(G[W])\cup\{wz_i:~w\in W \text{ and } w x_i,wy_{\sigma(i)}\in E(G)\}$.
\end{itemize}

For part (i), by deleting edges if necessary, we may assume that every vertex in $W$ has degree exactly $\delta m$ to both $V_1$ and $V_2$. Fix $w\in W$, define a random variable $\textbf{X}_{\sigma}(w):=|N_{\Gamma_\sigma}(w)\cap Z|.$ 
  Without loss of generality, assume that $N_{G}(w)\cap V_1=\{x_1,\ldots,x_{\delta m}\}$ and $N_{G}(w)\cap V_2=\{y_1,\ldots,y_{\delta m}\}$. 
  When $\textbf{X}_{\sigma}(w)=k$, we see that $|\{\sigma(1),\ldots,\sigma(\delta m)\}\cap [\delta m]| = k$ and so 
  $$\mathbb{P}(\textbf{X}_{\sigma}(w)=k)=\frac{\binom{\delta m}{k}\binom{m-\delta m}{\delta m-k}}{\binom{m}{\delta m}}.$$
 That is, $\textbf{X}_{\sigma}(w)\sim \mathrm{Hyp}(m,\delta m,\delta m)$. Then by \cref{lmm:tail} (setting $N=m$, $K=n=\delta m$, $p=\delta$ and $t=\delta/2$), we get
\begin{align*}
   \mathbb{P}(\textbf{X}_{\sigma}(w)\le \delta^2 m/2 ) \le e^{-\delta^3 m/2}\le 1/m^2. 
\end{align*}
 By the union bound, we can fix a permutation $\tau$ such that $\textbf{X}_{\tau}(w)\ge \delta^2|Z|/2$ for every $w\in W$. 
 By \cref{lmm:0-1/2} with $(X,Y)=(W,Z)$, we conclude that $\Gamma_\tau$ contains a copy of $Z^t_k(A,U)$ with $(A,U) \subseteq (W,Z)$. By the definition of $\Gamma_{\tau}$, this implies that $G$ contains two copies of $Z^t_k$ that share the same part $A$, denoted as $Z^t_k(A,U_1)$ and $Z^t_k(A,U_2)$ with $(A,U_1,U_2)\subseteq(W,V_1,V_2)$ as desired.

For part (ii), note that from the density conditions, any two vertices in $W$ have at least $2\delta m$ common neighbors in both $V_1$ and $V_2$. Fix distinct $w_i,w_j\in W$ and define
\[
\textbf{X}_{\sigma}(w_i,w_j):=|N_{\Gamma_\sigma}(w_i)\cap N_{\Gamma_\sigma}(w_j)\cap Z|.
\]
Choose subsets
\[
A_{ij}\subseteq N_G(\{w_i,w_j\})\cap V_1,\qquad 
B_{ij}\subseteq N_G(\{w_i,w_j\})\cap V_2
\]
with $|A_{ij}|=|B_{ij}|=2\delta m$. The number of indices matched by $\sigma$ from $A_{ij}$ to $B_{ij}$ has distribution $\mathrm{Hyp}(m,2\delta m,2\delta m)$ and is a lower bound for $\textbf{X}_{\sigma}(w_i,w_j)$. Thus by \cref{lmm:tail}, we get
\begin{align*}
   \mathbb{P}(\textbf{X}_{\sigma}(w_i,w_j) \le 2\delta^2 m ) \le e^{-4\delta^3 m}\le 1/m^3. 
\end{align*}

By the union bound, there exists a permutation $\tau$ such that for every pair
$w_i,w_j\in W$,
\[
\textbf{X}_{\tau}(w_i,w_j)\ge 2\delta^2|Z|.
\]
Thus $(W,Z)$ is $2\delta^2$-pair-dense in $\Gamma_\tau$. By \cref{lmm:0-1/2} applied to $\Gamma_\tau$ with $(X,Y)=(W,Z)$, we obtain a copy of $\mathcal F_k^t$ in $G[W]$, say $T$, such that
\[
|N_{\Gamma_\tau}(T)\cap Z|\ge t.
\]
By the definition of $\Gamma_\tau$, this implies
\[
|N_G(T)\cap V_1|\ge t
\quad\text{and}\quad
|N_G(T)\cap V_2|\ge t.
\]

Part (iii) follows directly from \cref{lmm:0-1/2} with $(X,Y,Z)=(W,V_1,V_2)$ since $(W,V_2)$ is $\delta$-pair-dense.
\end{proof}

\section{Proof of Theorem \ref{thm:main}}\label{sec: prof}
Let $H_1$ and $H_2$ be two graphs with $\chi(H_1)=\chi(H_2) =3$. Fix real numbers $0 \ll \varepsilon_1 \ll \varepsilon_2 \ll d \ll \gamma$. 
Suppose $G$ is an $n$-vertex graph with an $(H_1,H_2)$-free 2-edge-coloring and $\delta(G)\ge  (f(H_1,H_2)+\gamma)n$, in which $G_r$ is $H_1$-free and $G_b$ is $H_2$-free
(Recall that $f(H_1,H_2)$ is the value assigned by the decision procedure summarized in \cref{Figure:main0}). We need to show that there exists a constant $C_0= C_0(H_1,H_2,\gamma)$ such that $\chi(G)\le C_0$.

\subsection{Setup}

Applying the Multicolor Regularity Lemma (\cref{thm:regularity}) to $G$, we get a partition of $V(G)$ into sets $V_0,V_1, \ldots, V_k$ with $|V_0|\le \varepsilon_1 n$, $|V_1|=\cdots=|V_k|$, for some $1 / \eps_1<k<M$ such that all except $\eps_1 k^{2}$ of the pairs $V_{i},V_{j}$, $1\leq i<j\leq k$, are $\eps_1$-regular with respect to both $G_r$ and $G_b$. Let $R_r$ and $R_b$ be $(\eps_1,d)$-reduced graphs of $G_r$ and $G_b$ respectively, and let $R=R_r\cup R_b$. We call an edge $e$ in $R$ red (or blue resp.) if $e\in E(R_r)$ (or if $e\in E(R_b)$ resp.). Note that an edge in $R$ may be both red and blue.

We first state some basic properties of $R$. As $\chi(H_1)=\chi(H_2)=3$, by Embedding Lemma (\cref{thm:embedding-lemma}), we have the following. 
\begin{claim}\label{cl:triangle-free}
$R_r$ and $R_b$ are  both triangle-free.
\end{claim}

Let $L$ be the set of vertices $v$ in $V(R)=[k]$ with $d_{R}(v)\le (f(H_1,H_2) + \gamma /2)k$. The following claim is a standard fact in the minimum degree version of regularity lemma. We include a full proof of multicolor case here for completeness.

\begin{claim}\label{claim:L}
$|L| < \gamma k/4$.
\end{claim}

\begin{poc}
On the one hand, by the minimum degree condition, we have 
\begin{align*}
    \sum_{i\in L}\sum_{v\in V_i} d_G(v)\ge |L|\frac{(1-\varepsilon_1)n}{k}(f(H_1,H_2)+\gamma)n.
\end{align*}

On the other hand, 
$$
\sum_{i\in L}\sum_{v\in V_i} d_G(v)
    \le 2\sum_{i\in L}e(V_i)+\sum_{i\in L}e(V_i,V_0)+\sum_{i\in L}\sum_{ij\in E(R)} e(V_i,V_j)+\sum_{i\in L}\sum_{ij\notin E(R)} e(V_i,V_j). $$
Let us bound these four terms separately.
\begin{itemize}
\item Since $|V_1|=\cdots=|V_k|\le n/k$, we have $2\sum_{i\in L}e(V_i)\le 2|L|\binom{n/k}{2}$.

\item Since $|V_0|\le \varepsilon_1n$, $\sum_{i\in L}e(V_i,V_0)\le |L|\varepsilon_1 n \frac{n}{k}$.

\item For any $i\in L$, there are at most $\big(f(H_1,H_2)+ \gamma /2\big)k$ choices of $j$ such that $ij\in E(R)$. Hence, $\sum_{i\in L}\sum_{ij\in E(R)} e(V_i,V_j)\le |L|\Big(f(H_1,H_2)+\frac{\gamma}{2}\Big)k  \left(\frac{n}{k}\right)^2.$

\item For any $ij\notin E(R)$, either $(V_i,V_j)$ is not $\varepsilon_1$-regular or $d(V_i,V_j)\le 2d$. Recall that there are at most $\varepsilon_1 k^2$ pairs $(V_i,V_j)$ in the first case, thus $\sum_{i\in L}\sum_{ij\notin E(R)} e(V_i,V_j)\le \varepsilon_1 k^2\left(\frac{n}{k}\right)^2+|L|k 2d\left(\frac{n}{k}\right)^2.$
\end{itemize}
Putting them together, we have
\begin{align*}
   &|L|\frac{(1-\varepsilon_1)n^2}{k}(f(H_1,H_2)+\gamma)\le \sum_{i\in L}\sum_{v\in V_i} d_{G}(v)\\
    \le& 2|L|\binom{n/k}{2}+\varepsilon_1 n |L|\frac{n}{k}+|L|\Big(f(H_1,H_2)+\frac{\gamma}{2}\Big)k  \left(\frac{n}{k}\right)^2+2|L|k d\left(\frac{n}{k}\right)^2+\varepsilon_1 k^2\left(\frac{n}{k}\right)^2\\
    \le& |L|\frac{n^2}{k}(f(H_1,H_2)+3\gamma/4)+\varepsilon_1n^2.
\end{align*}

This simplifies to $((1/4-\varepsilon_1)\gamma-\varepsilon_1f(H_1,H_2))|L|\le \varepsilon_1k$. Hence
$$|L|\le \frac{\varepsilon_1 k}{(1/4-\varepsilon_1)\gamma-\varepsilon_1f(H_1,H_2)}\le \frac{\varepsilon_1k}{0.2\gamma}<\frac{\gamma k}{4},$$
as desired.
\end{poc}

\begin{claim}\label{cl:independent-set}
$\alpha(R)<(1-f(H_1,H_2))k$.
\end{claim}

\begin{poc}
Let $R'$ be the subgraph of $R$ induced by $[k]\setminus L$. Then by \cref{claim:L}, \[\delta(R')\ge (f(H_1,H_2)+ \gamma /2)k-|L|\ge (f(H_1,H_2)+ \gamma /4)k.\] Hence, $\alpha(R)\le \alpha(R')+|L|\le |R'|-\delta(R')+|L|<(1-f(H_1,H_2))k$.
\end{poc}

\subsection{Partition via color density vectors}

For each $v\in V(G)$, we define the \textit{red neighbor density} and \textit{blue neighbor density} of $v$ in $V_i$ as follows: \[d_r(v,V_i)=\frac{|N_{G_r}(v)\cap V_i|}{|V_i|} \quad\text{and} \quad d_b(v,V_i)=\frac{|N_{G_b}(v)\cap V_i|}{|V_i|}.\]
We write $d(v,V_i)=d_r(v,V_i)+d_b(v,V_i)$.

Let $\mathcal{S}=\{0,\frac{1}{2},1\}^{2k}$. Define a map $\phi: V(G) \to \mathcal{S}$ with $v\mapsto(\phi_r^1(v),\ldots,\phi_r^k(v),\phi_b^{1}(v),\ldots,\phi_b^{k}(v))$ where
\begin{align*}
\phi_r^i(v) = 
\begin{cases}
0, & \text{if } d_r(v,V_i) < \varepsilon_2 \\
\frac{1}{2}, & \text{if } \varepsilon_2 \leq d_r(v,V_i) < \frac{1}{2}+\varepsilon_2 \\
1, & \text{if } d_r(v,V_i)\ge \frac{1}{2}+\varepsilon_2
\end{cases}
\quad \text{ and } \quad 
\phi_b^i(v) = 
\begin{cases}
0, & \text{if } d_b(v,V_i) < \varepsilon_2 \\
\frac{1}{2}, & \text{if } \varepsilon_2 \leq d_b(v,V_i) < \frac{1}{2}+\varepsilon_2 \\
1, & \text{if } d_b(v,V_i)\ge \frac{1}{2}+\varepsilon_2.
\end{cases}
\end{align*}
Partition $V(G)$ by considering the red/blue neighbor density of vertices in $V_i$'s as follows:
\begin{align*}
 V(G)=\bigcup_{\bm{v}\in\mathcal{S}}X_{\bm{v}} \quad \text{where} \quad X_{\bm{v}}=\{v\in V(G):~ \phi (v)=\bm{v}\}.
\end{align*}

Take $C$ sufficiently large, in particular larger than all constants required by Lemmas~\ref{lmm:embed key}, \ref{lmm:xiaobaima}, \ref{lmm:0-1/2} and the preceding regularity setup.
It suffices to show that for any $\bm{v}=(r_1,\ldots,r_k,b_1,\ldots,b_k)$, $\chi(G[X_{\bm{v}}])\le C$ as
\[\chi(G)\le  \sum_{\bm{v}\in\mathcal{S}}\chi(G[X_{\bm{v}}]).\]
To this end, fix an arbitrary $\bm{v}=(r_1,\ldots,r_k,b_1,\ldots,b_k)$ and suppose to the contrary that $\chi(G[X_{\bm{v}}]) > C$. 
Let $I^{+}=I^+(\bm{v})=I_r^{+}\cup I_{b}^{+}\subseteq [k]\setminus L$, where 
$$ I_r^{+}=I_r^{+}(\bm{v})=\{i\in [k] \setminus L:~r_i>0\} \quad \text{and} \quad I_b^{+}=I_b^{+}(\bm{v})=\{i\in [k] \setminus L:~b_i>0\}$$
are the sets of indices in $[k] \setminus L$ with positive red or blue density to $X_{\bm{v}}$ respectively. Note that some index might lie in both $I_r^{+}$ and $I_{b}^{+}$.

We next observe some basic properties of this  partition.
\begin{claim}\label{claim:degree}
For any $x\in X_{\bm{v}}$, we have $d(x)< \frac{n}{k}\sum_{i\in I^{+}}d(x,V_i)+\frac{\gamma}{2}n.$
\end{claim}

\begin{poc}
 Note that for any $i\in [k] \setminus (L\cup I^{+})$, $r_i=b_i=0$. Thus, $d_r(x,V_i)<\varepsilon_2$, $d_b(x,V_i)<\varepsilon_2$, which infers that $d(x,V_i)<2\varepsilon_2$. Hence
\begin{align*}
d(x)
&=\sum_{i=0}^k d(x,V_i)=d(x,V_0)+ \sum_{i\in L}d(x,V_i)+\sum_{i\in I^{+}}d(x,V_i)+\sum_{i\in [k] \setminus (L\cup I^{+})}d(x,V_i)\\
&< |V_0|+|L|\frac{n}{k}+\frac{n}{k}\sum_{i\in I^{+}}d(x,V_i)+2\varepsilon_2n\\
&\le \frac{n}{k}\sum_{i\in I^{+}}d(x,V_i)+\frac{\gamma}{2}n \quad(\text{By \cref{claim:L}}),
\end{align*}
as desired.
\end{poc}

\begin{claim}\label{claim:noin} There is no red edge in $R[I_r^{+}]$ and no  blue edge in $R[I_b^{+}]$.
\end{claim}

\begin{poc}
Suppose to the contrary that $ij$ is a red edge in $R[I_r^{+}]$. We shall prove that there is a red copy of $K_3[|H_1|]$ in $(V_i,V_j,X_{\bm{v}})$, but then $H_1\subseteq G_r$ as $\chi(H_1)=3$, a contradiction. Recall that $|X_{\bm{v}}|\ge \chi(G[X_{\bm{v}}]) > C$ is sufficiently large and $i,j\in I_r^+$, hence apply \cref{lmm:in-ex} twice: first to the collection of sets $\{N_{G_r}(v)\cap V_i: v\in X_{\bm{v}}\}$, and then to the resulting sub-collection to get a subset $X'\subseteq X_{\bm{v}}$ of size $|H_1|$ such that $|N_{G_r}(X')\cap V_i|\ge \eps_1|V_i|$ and $|N_{G_r}(X')\cap V_j|\ge \eps_1|V_j|.$

Since $ij$ is a red edge and $(V_i,V_j)$ is $\eps_1$-regular, the red graph between $N_{G_r}(X')\cap V_i$ and $N_{G_r}(X')\cap V_j$ has density at least $d-\varepsilon_1$ and hence contains a red $K_{|H_1|,|H_1|}$, which together with $X'$ yields a desired red $K_3[|H_1|]$. The argument for blue color is identical.
\end{poc}

\begin{claim}\label{cor:intersection}
For any $u\in V(R)$, $N_{R_r}(u)\cap I_b^+$ and $N_{R_b}(u)\cap I_r^+$ are independent sets in $R$, hence of size at most $(1-f(H_1,H_2))k$. 
\end{claim} 
\begin{poc}
Let us examine that $N_{R_r}(u)\cap I_b^+$ is an independent set and then the claimed upper bound on its cardinality follows from \cref{cl:independent-set}. According to \cref{cl:triangle-free}, $R_r$ is triangle-free, implying that there is no red edge within $N_{R_r}(u)$. By \cref{claim:noin}, there is no blue edge within $I_b^+$, hence $N_{R_r}(u)\cap I_b^+$ is an independent set. The argument for $N_{R_b}(u)\cap I_r^+$ is identical.
\end{poc}

\begin{claim}\label{claim:size}
$\max\{|I_r^{+}|,|I_b^{+}|\} \le 2(1-f(H_1,H_2))k.$
\end{claim}

\begin{poc}
Suppose to the contrary that $|I_r^{+}| > 2(1- f(H_1,H_2))k$. By \cref{claim:noin}, there are no red edges in $R[I_r^{+}]$, meaning $R[I_r^{+}]=R_b[I_r^{+}]$. By \cref{cl:triangle-free}, $R_b[I_r^{+}]$ is triangle-free. Hence, by Mantel's theorem, the minimum degree of $R_b[I_r^{+}]$ is at most $|I_r^{+}|/2$. Consequently, there exists a vertex $v\in I_r^{+}$ with at least $|I_r^{+}|/2-1>(1- f(H_1,H_2))k-1$ non-neighbors within $I_r^{+}$. This implies that \[d_R(v) \le (k-1)-((1- f(H_1,H_2))k-1) \le f(H_1,H_2)k,\] contradicting the fact that $v\in I_r^{+}\subseteq [k]\setminus L$. Similarly, we have  $|I_b^{+}| \le 2(1- f(H_1,H_2))k$.
\end{poc}

To provide a more detailed analysis of the structure of $R$, we introduce the following definitions:
\begin{itemize}
\item  let $I_r^1$ be the set of indices in $[k]\setminus L$ such that $r_i =1$;
\item  let $I_b^{*}=I_b^{+}\setminus I_r^1$ be the set of indices in $[k]\setminus (L\cup I_r^1)$ such that $b_i >0$;
\item  let $I_r^{*}=I_r^{+}\setminus (I_r^1\cup I_b^{*})$ be the set of indices in $[k]\setminus (L\cup I_r^1\cup I_b^{*})$ such that $r_i =1/2$;
\item  let $D=[k]\setminus (I_r^1\cup I_b^*\cup I_r^*)$.
\end{itemize}  
Note that $L\subseteq D$, $I_r^*\cup I_r^1\subseteq I_r^+$ and $I^+=I_r^1\cup I_b^*\cup I_r^*$.

\begin{claim}\label{cl:huoyibu}
$|I_r^{*}|/2+ |I_r^1|+ |I_b^{*}| > \big(f(H_1,H_2)+\frac{\gamma}{4}\big)k$ and  $|I_r^1|+ |I_b^{*}|\ge \big(6f(H_1,H_2)-4+\frac{\gamma}{2}\big)k.$
\end{claim}
\begin{poc}
Fix $x\in X_{\bm{v}}$, by \cref{claim:degree}, we have 
\begin{align*}
   d(x)
   &< \frac{n}{k}\sum_{i\in I^{+}}d(x,V_i)+\frac{\gamma}{2}n=\frac{n}{k}\big(\sum_{i\in I_r^{*}}d(x,V_i)+\sum_{i\in I_r^{1}}d(x,V_i)+\sum_{i\in I_b^{*}}d(x,V_i)\big)+\frac{\gamma}{2}n\\
   &\le \frac{n}{k}\big((1/2+2\varepsilon_2)|I_r^{*}|+|I_r^1|+|I_b^{*}|\big)+\frac{\gamma}{2}n.
\end{align*}
  This, together with the minimum degree of $G$, implies that 
\begin{align*}
   |I_r^{*}|/2+ |I_r^1|+ |I_b^{*}| > (f(H_1,H_2)+\gamma/4)k. 
\end{align*}
Hence by \cref{claim:size},
\begin{align*}
|I_r^1|+ |I_b^{*}|&=2(|I_r^{*}|/2+ |I_r^1|+ |I_b^{*}|)-(|I_r^{*}|+|I_r^{1}|)-|I_b^{*}|\\
&\ge 2(f(H_1,H_2)+\gamma/4)k-4(1-f(H_1,H_2))k= \Big(6f(H_1,H_2)-4+\frac{\gamma}{2}\Big)k,
\end{align*}
as claimed.
\end{poc}

Now we are prepared to prove \cref{thm:main}. We split the proof into two parts depending on whether one of the $H_i$ is contained in a blowup of $C_5$.

\subsection{Neither $H_1$ nor $H_2 $ embeds in blowups of $C_5$}\label{case2}

The proof will be divided into cases corresponding to the characterization in \cref{Figure:main0}. 
Recall that $\chi(G_r[X_{\bm{v}}])\cdot\chi(G_b[X_{\bm{v}}])\ge \chi(G[X_{\bm{v}}])>C$. Throughout this subsection, the blue and red colors are symmetric, so we suppose that $\chi(G_r[X_{\bm{v}}])>\sqrt{C}\gg |H_1|,|H_2|,1/\gamma$. 

\begin{case}
 $f(H_1,H_2) = 4/5$.
\end{case}
For this case, we do not need any structural information about $H_1,H_2$ and use only the fact that $\chi(H_1)=\chi(H_2)=3$.
Since $|L| \le \gamma k/4$ and for any $v\in [k]\setminus L$, $d(v) \ge (4/5+\gamma/2)k$,  $R$ contains $K_6$ by Tur\'an's theorem. Since $R(3,3)=6$, $R$ contains a monochromatic triangle, contradicting \cref{cl:triangle-free}.

\begin{case}
$f(H_1,H_2) =7/9$. 
\end{case}
In this case, note that both $H_1,H_2\hookrightarrow \C Z_0$. By  \cref{claim:noin}, there are no blue edges in $R[I_b^{*}]$. If $I_b^{*}$ forms an independent set in $R$, then by \cref{cl:independent-set}, $|I_b^{*}|\le \frac{2}{9}k.$ Combining with~\cref{cl:huoyibu} and~\cref{claim:size}, we get that 
$$7k/9<|I_r^*|+|I_r^1|+|I_b^*|\le |I_r^+|+2k/9\le 2(1-7/9)k+2k/9 =2k/3,$$ 
a contradiction.  Hence, we can fix a red edge $uv\in R[I_b^{*}]$.
\begin{claim}\label{cl:no-P4-in-Ib*}
   We cannot have both $u$ and $v$ sending red edges to $I_r^{1}$.
\end{claim}

\begin{poc}
  Suppose, for contradiction, that there exist $uu',vv'\in E(R_r)$ with
  $u',v'\in I_r^{1}$.
  Since $R_r$ is triangle-free, we have $u'\neq v'$.
  Let $V_1,V_2,V_3,V_4$ be the clusters corresponding to
  $u',u,v,v'$, respectively. Then
  \[
  V_1V_2,\quad V_2V_3,\quad V_3V_4
  \]
  are red edges in the reduced graph. Moreover, since $u',v'\in I_r^1$,
  the pairs $(X_{\bm v},V_1)$ and $(X_{\bm v},V_4)$ are both
  $(1/2+\varepsilon_2)$-dense in red.
  Since $H_1\hookrightarrow\C Z_0$, \cref{lmm:embed key}(i)
  gives a red copy of $H_1$, a contradiction.
\end{poc}

We may then assume that $u$ has no red neighbors in $I_r^{1}$.
  Let $A_1$ be the set of blue neighbors of $u$ in $I_r^{1}$ and $A_2$ be the set of red neighbors of $u$ in $I_b^{*}$. Then $A_1\subseteq N_{R_b}(u)\cap I_r^+$ and $A_2\subseteq N_{R_r}(u)\cap I_b^+$.
  Hence by~\cref{cor:intersection}, $|A_1|,|A_2|\le  2k/9$.
  As $u\in [k]\setminus L$, we have
\begin{align*}
    \Big(\frac{7}{9}+\frac{\gamma}{2}\Big)k\le d_R(u) \le |A_1|+|A_2|+(k-|I_r^{1}|-|I_b^{*}|)\le \frac{13k}{9}-|I_r^{1}|-|I_b^{*}|.
\end{align*}
This implies that $|I_r^{1}|+ |I_b^{*}|\le \big(\frac{2}{3}-\frac{\gamma}{2}\big)k<(6f(H_1,H_2)-4)k,$  contradicting \cref{cl:huoyibu}.

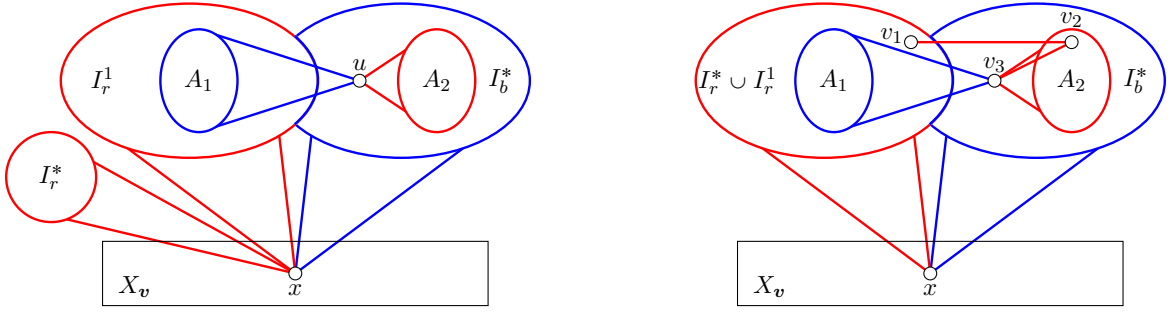
\begin{figure}[H]
    \centering

    \begin{subfigure}[t]{0.48\textwidth}
        \centering
        \begin{tikzpicture}[scale=0.85, transform shape]
          \coordinate (a) at (3,0);
          \coordinate (a1) at (1.35-1.7,2.5);
          \coordinate (a2) at (1.35+2,3);
          \coordinate (a3) at (4.65-2,3);
          \coordinate (a4) at (4.65+1.7,2.5);
          \draw [red,line width=0.901pt] (a) -- (a1);
          \draw [blue,line width=0.901pt] (a) -- (a2);
          \draw [red,line width=0.901pt] (a) -- (a3);
          \draw [blue,line width=0.901pt] (a) -- (a4);

          \filldraw [blue,line width=0.901pt, fill=white] (4.65,3) ellipse (2 and 1.2);
          \filldraw [red,line width=0.901pt, fill=white] (1.35,3) ellipse (2 and 1.2);

          \draw [red,line width=0.901pt] (3,0) -- (-0.8,2+0.1);
          \draw [red,line width=0.901pt] (3,0) -- (-0.8,1-0.1);
          \filldraw [red,line width=0.901pt, fill=white] (-0.8,1.5) ellipse (0.7 and 0.7);
          \node at (-0.8,1.5) {$I_r^{*}$};
          
          \draw[white, line width=1pt] (2.99,2.32) arc (-35:35:2 and 1.2);
          \draw[blue, line width=1pt] (2.99,2.32) arc (-35:35:2 and 1.2);

          \coordinate (u) at (4,3);
          \coordinate (x1) at (1.5,3.8);
          \coordinate (x2) at (1.5,2.2);
          \draw [blue,line width=0.901pt] (u) -- (x1);
          \draw [blue,line width=0.901pt] (u) -- (x2);
          \coordinate (y1) at (5.2,3.8);
          \coordinate (y2) at (5.2,2.2);
          \draw [red,line width=0.901pt] (u) -- (y1);
          \draw [red,line width=0.901pt] (u) -- (y2);

          \draw (0,-0.5) rectangle (6,0.5);
          
          \filldraw [blue,line width=0.901pt, fill=white] (1.5,3) ellipse (0.6 and 0.8);
          \filldraw [red,line width=0.901pt, fill=white] (5.2,3) ellipse (0.6 and 0.8);

          \filldraw[fill=white] (4,3) circle (0.1cm);
          \filldraw[fill=white] (3,0) circle (0.1cm);
          
          \node at (4,3.25) {$u$};
          \node at (1.5,3) {$A_1$};
          \node at (5.2,3) {$A_2$};
          \node at (3,-0.25) {$x$};
          \node at (0.5,-0.2) {$X_{\bm{v}}$};
          \node at (0,3) {$I_r^{1}$};
          \node at (6.2,3) {$I_b^{*}$};
        \end{tikzpicture}
        \caption{Illustration of Case 2. We consider clusters as vertices in $R$ and disregard clusters with low density with respect to $x$.}
        \label{fig:case2}
    \end{subfigure}
    \hfill
    \begin{subfigure}[t]{0.48\textwidth}
        \centering
        \begin{tikzpicture}[scale=0.85, transform shape]
          \coordinate (a) at (3,0);
          \coordinate (a1) at (1.35-1.7,2.5);
          \coordinate (a2) at (1.35+2,3);
          \coordinate (a3) at (4.65-2,3);
          \coordinate (a4) at (4.65+1.7,2.5);
          \coordinate (v1) at (2.7,3.6);
          \coordinate (v2) at (5.2,3.6);
          \draw [red,line width=0.901pt] (a) -- (a1);
          \draw [blue,line width=0.901pt] (a) -- (a2);
          \draw [red,line width=0.901pt] (a) -- (a3);
          \draw [blue,line width=0.901pt] (a) -- (a4);
          
          \filldraw [blue,line width=0.901pt, fill=white] (4.65,3) ellipse (2 and 1.2);
          \filldraw [red,line width=0.901pt, fill=white] (1.35,3) ellipse (2 and 1.2);

          \draw[white, line width=1pt] (2.99,2.32) arc (-35:35:2 and 1.2);
          \draw[blue, line width=1pt] (2.99,2.32) arc (-35:35:2 and 1.2);
          
          \coordinate (u) at (4,3);
          \coordinate (x1) at (1.5,3.8);
          \coordinate (x2) at (1.5,2.2);
          \draw [blue,line width=0.901pt] (u) -- (x1);
          \draw [blue,line width=0.901pt] (u) -- (x2);
          \coordinate (y1) at (5.2,3.8);
          \coordinate (y2) at (5.2,2.2);
          \draw [red,line width=0.901pt] (u) -- (y1);
          \draw [red,line width=0.901pt] (u) -- (y2);

          \draw (0,-0.5) rectangle (6,0.5);
          
          \filldraw [blue,line width=0.901pt, fill=white] (1.5,3) ellipse (0.6 and 0.8);
          \filldraw [red,line width=0.901pt, fill=white] (5.2,3) ellipse (0.6 and 0.8);

          \filldraw[fill=white] (3,0) circle (0.1cm);
          \draw [red,line width=0.901pt] (v2) -- (v1);
          \draw [red,line width=0.901pt] (v2) -- (u);
          \filldraw[fill=white] (v1) circle (0.1cm);
          \filldraw[fill=white] (4,3) circle (0.1cm);
          \filldraw[fill=white] (v2) circle (0.1cm);
          
          \node at (2.4,3.7) {$v_1$};
          \node at (5.2,3.95) {$v_2$};
          \node at (4,3.25) {$v_3$};
          \node at (1.5,3) {$A_1$};
          \node at (5.2,3) {$A_2$};
          \node at (3,-0.25) {$x$};
          \node at (0.5,-0.2) {$X_{\bm{v}}$};
          \node at (0,3) {$I_r^*\cup I_r^1$};
          \node at (6.2,3) {$I_b^{*}$};
        \end{tikzpicture}
        \caption{Illustration of Case 3. Here, $v_1\in I_r^1$.}
        \label{fig:case3}
    \end{subfigure}

    \caption{Illustrations of Cases 2 and 3.}
\end{figure}
 
\begin{case}
    $f(H_1,H_2) =3/4$.
\end{case}

 In this case, both $H_1,H_2\hookrightarrow \C Z_1$. By \cref{claim:size}, we have $|I_b^{*}|\le \frac{1}{2}k$, while we have  $|I_r^{1}|+|I_b^{*}|\ge (\frac{1}{2}+\frac{\gamma}{2})k$ by \cref{cl:huoyibu}. Hence,  $I_r^{1}\neq \varnothing$.

\begin{claim}\label{claim:sianke}
For any vertex $v\in I_r^{1}$, $v$ has no red neighbors in $I_b^{*}$.
\end{claim}

\begin{poc}
Suppose to the contrary that there exists a red edge $v_1v_2\in E(R_r)$ with $v_1\in I_r^1$ and $v_2\in I_b^{*}$.  
As $v_2\in [k]\setminus L$, it follows that $d_R(v_2) \ge (3/4+\gamma/2)k$. By \cref{cl:huoyibu} and \cref{claim:size}, we obtain
\begin{align*}
|I_b^{*}|\ge  \Big(\frac{3}{4}+\frac{\gamma}{4}\Big)k-(|I_r^{*}|/2 + |I_r^1|)\ge \Big(\frac{3}{4}+\frac{\gamma}{4}\Big)k-|I_r^{*}\cup I_r^1|\ge \Big(\frac{1}{4}+\frac{\gamma}{4}\Big)k.
\end{align*}
Since $d_R(v_2)>k-|I_b^{*}|$, we can find $v_3\in N_R(v_2)\cap I_b^{*}$. By \cref{claim:noin}, $v_2v_3$ must be a red edge.

We will show that $v_3$ has no red neighbors in $I_r^{*}\cup I_r^{1}$. Otherwise, assume that there exists $v_4\in I_r^{*}\cup I_r^{1}$ with $v_3v_4\in E(R_r)$. As $R_r$ is triangle-free, $v_4\neq v_1$. Apply \cref{lmm:embed key}(ii) to the red path with clusters ordered as
$v_4,v_3,v_2,v_1$. More precisely, let $V_1,V_2,V_3,V_4$ be the clusters
corresponding to $v_4,v_3,v_2,v_1$, respectively. Then
$(X_{\bm v},V_1)$ is $\varepsilon_2$-dense and $(X_{\bm v},V_4)$ is
$(1/2+\varepsilon_2)$-dense in red. Since $H_1\hookrightarrow\C Z_1$,
\cref{lmm:embed key}(ii) gives a red copy of $H_1$, a contradiction.

By \cref{cl:huoyibu} we can deduce that $|D|=k-(|I_r^{*}|+|I_r^1|+|I_b^{*}|)<(\frac{1}{4}-\frac{\gamma}{4})k.$ Hence by \cref{cor:intersection} and that $v_3\in [k]\setminus L$, we obtain
\begin{equation}\label{eq:non-L-vertex}
  \Big(\frac{3}{4}+\frac{\gamma}{2}\Big)k\le d_R(v_3)\le |N_{R_r}(v_3)\cap I_b^+|+|N_{R_b}(v_3)\cap I_r^+|+|D|\le \Big(\frac{3}{4}-\frac{\gamma}{4}\Big)k,
\end{equation}  
a contradiction.
\end{poc}

Choose $v_1\in  I_r^{1}$.   According to \cref{claim:noin} and \cref{claim:sianke}, $v_1$ has no red neighbors in $I_r^{+}$ and $I_b^{*}$. Therefore, the red neighbors of $v_1$ can only be in $D$. That is, $N_{R_r}(v_1)\subseteq D$. 
  
  We now consider the blue neighbors of $v_1$ out of $D$.
  Define $A_3:=N_{R_b}(v_1)\setminus D.$
  By \cref{cl:triangle-free}, there are no blue edges in $A_3$. Thus, $R[A_3]=R_r[A_3]$ is triangle-free.
  Therefore, there exists a vertex $w\in A_3$ such that $|N_R(w)\cap A_3|\le |A_3|/2$. 
  Since $w\in [k]\setminus D\subseteq [k]\setminus L$, we have
\begin{align*}
   \Big(\frac{3}{4}+\frac{\gamma}{2}\Big)k\le  d_R(w) \le |N_R(w)\cap A_3| + (k-|A_3|)\le k-\frac{|A_3|}{2},
\end{align*}
which infers that $|A_3|\le (1/2-\gamma)k$. As $v_1\in[k]\setminus L$ and recall that $N_{R_r}(v_1)\subseteq D$, by \cref{claim:L}, we have
\begin{align*}
   \Big(\frac{3}{4}+\frac{\gamma}{2}\Big)k\le d_R(v_1)\le |D|+|A_3| < \frac{3k}{4},
\end{align*}
which contradicts \cref{claim:L}.

\begin{case}
$f(H_1,H_2) = 5/7$. 
\end{case}

  In this case, both $H_1,H_2\hookrightarrow \C Z_2$ and $\delta_{\chi}(H_1),\delta_{\chi}(H_2)\le 1/3$. If there exists an index $i\in [k]$ such that $r_i=1$, then  for any $u\in X_{\bm{v}}$, $|N(u)\cap V_i|\ge (1/2+\varepsilon_2) |V_i|$. Hence for any $u,v\in X_{\bm{v}}$, $|N(u)\cap N(v)\cap V_i|\ge 2\varepsilon_2 |V_i|$. As $\chi(G_r[X_{\bm{v}}])$ is large, we can apply \cref{lmm:0-1/2} to get a copy of $H_1$ in $G_r[X_{\bm{v}}\cup V_i]$, a contradiction. Thus, $I_r^1=\varnothing$ and $I_b^*=I_b^+$. Note also that \cref{claim:size} tells us $\max\{|I_r^{*}|,|I_b^{+}|\}\le \frac{4}{7}k$. Combining all these with \cref{cl:huoyibu}, we have  
  $$|I_r^{*}|> 2\left(\Big(f(H_1,H_2)+\frac{\gamma}{4}\Big)k-|I_b^*|\right) \ge \Big(\frac{2}{7}+\frac{\gamma}{2}\Big)k$$ 
  and 
  $$|I_b^{+}|>\Big(f(H_1,H_2)+\frac{\gamma}{4}\Big)k-|I_r^*|/2\ge \Big(\frac{3}{7}+\frac{\gamma}{4}\Big)k.$$ 
  Then by \cref{cl:huoyibu} again, we have
  $$|I_r^{*}| + |I_b^{+}|=|I_r^{*}|/2+(|I_r^{*}|/2 + |I_b^{+}|)\ge \Big(\frac{6}{7}+\frac{\gamma}{2}\Big)k.$$ 
  Consequently,  $|D|=k-|I^+|=k - |I_r^{*}| - |I_b^{+}| \le (\frac{1}{7}-\frac{\gamma}{2})k$.
  
  It follows from \cref{claim:noin} that $R[I_b^+]$ contains no blue edges.
  If there is also no red edge in $R[I_b^+]$, then $I_b^+$ is an independent set of $R$.
  But then according to \cref{cl:independent-set}, $|I_b^+|\le \frac{2}{7}k$, a contradiction.

We may then assume that $R[I_b^+]$ contains a red edge, say $uv$. We claim that $u$ and $v$ cannot both have red neighbors in $I_r^{*}$. Indeed, suppose, for contradiction, that there exist $uu',vv'\in E[R_r]$ with $u',v'\in I_r^{*}$. According to \cref{cl:triangle-free}, $R_r$ is triangle-free, implying $u'\neq v'$.
Let $V_1$, $V_2$, $V_3$, $V_4$ denote the clusters corresponding to $u'$, $u$, $v$, $v'$, respectively.  Since $u',v'\in I_r^{*}$, $( X_{\bm{v}}, V_1)$ and $( X_{\bm{v}},V_4)$ are both $\varepsilon_2$-dense under $G_r$, by \cref{lmm:embed key}(iii) and $H_1\hookrightarrow \C Z_2$, we obtain a copy of $H_1$ in $G_r$, a contradiction.

Thus without loss of generality, assume $u$ has no red neighbors in $I_r^{*}$. Recall that $|D| \le (\frac{1}{7}-\frac{\gamma}{2})k$.
Similar to~\eqref{eq:non-L-vertex}, we get from \cref{cor:intersection} and $u\not\in L$ that
  \begin{align*}
 \Big(\frac{5}{7}+\frac{\gamma}{2}\Big)k\le d_R(u)\le |N_{R_r}(u)\cap I_b^+|+|N_{R_b}(u)\cap I_r^+|+|D|\le  \Big(\frac{5}{7}-\frac{\gamma}{2}\Big)k,
\end{align*}
a contradiction.  

\begin{case}\label{case5}
$f(H_1,H_2) = 2/3$.
\end{case}

 In this case $\delta_\chi(H_1)=\delta_\chi(H_2)=0$. Recall that $\chi(G_r[X_{\bm{v}}])$ is large. If there exists $i\in I_r^+$, then by \cref{lmm:0-1/2}, $H_1\subseteq G_r[X_{\bm{v}}\cup V_i]$, a contradiction. Thus, $I_r^+=\varnothing$ and then by \cref{cl:huoyibu}, we have $|I_b^{+}| > \big(\frac{2}{3}+\frac{\gamma}{4}\big)k$. However, \cref{claim:size} implies that $|I_b^{+}|\le 2(1-2/3)k=2k/3$, a contradiction.

\subsection{One of $H_1$ or $H_2$ embeds in blowups of $C_5$}\label{case1}
Assume without loss of generality that $H_1\hookrightarrow \C C_5$. Recall that we have $\chi(G_r[X_{\bm{v}}])\chi(G_b[X_{\bm{v}}])\ge \chi(G[X_{\bm{v}}])>C$ and so one of $\chi(G_r[X_{\bm{v}}])$ and $\chi(G_b[X_{\bm{v}}])$ is at least $\sqrt{C}$.

\begin{case}
 $f(H_1,H_2) = 3/4$.    
\end{case}
In this case, we use only the fact that $H_1\hookrightarrow \C C_5$.
Recall that $|L| \le \gamma k/4$, and for any $v\in [k]\setminus L$, $d_R(v) \ge (3/4+\gamma/2)k$.  
Hence, we can bound the size of $R$ as follows: 
$e(R)\ge \frac{1}{2}\big(1-\frac{\gamma}{4}\big)k\cdot\big(\frac{3}{4}+\frac{\gamma}{2}\big)k>\frac{3}{8}k^2.$
By Tur\'an's theorem, $R$ contains $K_5$ as a subgraph. By \cref{cl:triangle-free} this $K_5$ in $R$ contains no monochromatic triangles. Note that the only 2-edge-coloring of $K_5$ without monochromatic $K_3$ is the union of two edge-disjoint monochromatic $C_5$'s. Thus, there is a red $C_5$ and a blue $C_5$ in $R$. Hence as $H_1\hookrightarrow \C C_5$, we have $H_1\subseteq G_r$ by~\cref{thm:embedding-lemma}, a contradiction. 

\begin{case}
$f(H_1,H_2) = 5/7$.  
\end{case}
 In this case,  $\max\{\delta_\chi(H_1),\delta_\chi(H_2)\}=1/3$. If $\chi(G_r[X_{\bm v}])\ge \sqrt C$, then $r_i\le 1/2$ for every
$i\in[k]$. Indeed, if $r_i=1$ for some $i$, then every two vertices of
$X_{\bm v}$ have at least $2\varepsilon_2|V_i|$ common red neighbours in
$V_i$. Since $\delta_\chi(H_1)\le 1/3$ and $\chi(G_r[X_{\bm v}])$ is large,
\cref{lmm:0-1/2} gives a red copy of $H_1$ in $G_r[X_{\bm v}\cup V_i]$,
a contradiction.

 Hence $I_r^1=\varnothing$, and we have  $I_b^{*}=I_b^{+}$, $I_r^*=I_r^+\setminus I_b^*$ and $I^+=I_b^*\cup I_r^*$. 
By \cref{cl:huoyibu}, we have  
$|I_b^{+}|\ge|I_b^*|\ge (6f(H_1,H_2)-4+\gamma/2)k>(2/7+\gamma/4)k$. Moreover, combining this with \cref{claim:size} we have
\[
(5/7+\gamma/4)k<|I_r^{*}|/2+|I_b^{*}|\le \min\{|I_r^{*}|/2+4k/7, |I^{+}|/2+2k/7\}.
\]
Hence we also have  
$|I_r^{*}|>(2/7+\gamma/2)k$ and $|I^{+}|>(6/7+\gamma/2)k.$

Note that the above proof does not rely on the assumption that $H_1 \hookrightarrow \C C_5$. Hence, if $\chi(G_b[X_{\bm{v}}])\ge\sqrt{C}$, we may swap red and blue and repeat the above argument to obtain $|I_b^{+}|>(2/7+\gamma/4)k$ and $|I^{+}|>(6/7+\gamma/2)k$. Therefore, in either case (i.e., whether $\chi(G_r[X_{\bm{v}}])$ or $\chi(G_b[X_{\bm{v}}])$ is unbounded), we have
\[
|I^{+}|>(6/7+\gamma/2)k \quad \text{and} \quad |I_b^{+}|>(2/7+\gamma/4)k.
\]

We will show that  $|I^{+}|<6k/7$ to get a contradiction. Recall that $|I_b^{+}|>(2/7+\gamma/4)k$, so by \cref{cl:independent-set}, $I_b^{+}$ is not an independent set in $R$.  
According to \cref{claim:noin}, there are no blue edges within $I_b^{+}$, meaning that there exists at least one red edge $uv\in R[I_b^{+}]$.

\begin{claim}\label{cl:red-edge}
     If $uv$ is a red edge in $R[I_b^{+}]$, then $u$ and $v$ cannot both have red neighbors in $I_r^{+}$. 
\end{claim}

\begin{poc}
  Suppose, for contradiction, that there exist $uu',vv'\in E[R_r]$ with $u',v'\in I_r^{+}$.
  According to \cref{cl:triangle-free}, $R_r$ is triangle-free, implying that $u'\neq v'$.
  Let $V_u$, $V_v$, $V_{u'}$, $V_{v'}$ denote the clusters corresponding to $u$, $v$, $u'$, $v'$ respectively.

Since $|X_{\bm{v}}|\ge \chi(G[X_{\bm{v}}])>C$ is sufficiently large,  applying \cref{coro: large common}, there exists a subset $X'\subseteq X_{\bm{v}}$ with $|X'|=|H_1|=:t$ such that 
\begin{align*}
   |N_{G_r}(X')\cap V_{u'}|\ge \alpha|V_{u'}|, \quad |N_{G_r}(X')\cap V_{v'}|\ge \alpha|V_{v'}|,
\end{align*}   
where $\alpha=\alpha(t,\varepsilon_2)>0$. Pick $W_1\subseteq N_{G_r}(X')\cap V_{u'}$ and $W_4\subseteq N_{G_r}(X')\cap V_{v'}$
with $|W_1|=|W_4|=\alpha |V_{u'}|$.
Moreover, we pick  $W_2\subseteq V_u\setminus X'$ and $W_3\subseteq V_v\setminus X'$ with $|W_2|=|W_3|=\alpha |V_{u}|$.

By \cref{fact:slicinglemma}, for each $i\in[3]$, $(W_i,W_{i+1})$ is $(\varepsilon_1/\alpha,d-\varepsilon_1)$-regular. By \cref{thm:embedding-lemma}, there exists a $t$-blowup of $P_4$ with parts coming from $W_1$, $W_2$, $W_3$, $W_4$, respectively. Then, adding $X'$ we get a red copy of $C_5[t]$, which contains $H_1$, a contradiction.
\end{poc}

Fix a red edge $uv\in R[I_b^{+}]$. Without loss of generality, assume that $u$ has no red neighbors in $I_r^{+}$, i.e.~$N_{R_r}(u)\cap I_r^{+}=\varnothing$. Let 
$A_1=N_{R_r}(u)\cap I_b^{+}=N_{R_r}(u)\cap I^{+}$, 
$A_2=N_{R_b}(u)\cap I_r^{+}=N_{R_b}(u)\cap I^{+}$.

According to \cref{cor:intersection}, we have $|A_1|,|A_2|\le 2k/7$.  As $u\in I_b^{+}\subseteq [k]\setminus L$, we have
\[
\Big(\frac{5}{7}+\frac{\gamma}{2}\Big)k \le d_R(u)\le |N_R(u)\cap I^{+}|+|D|\le |A_1|+|A_2|+ (k-|I^{+}|)\le \frac{11k}{7}-|I^+|,
\]
implying that $|I^{+}|\le \big(6/7-\gamma/2\big)k$, a contradiction.

\begin{case}
   $f(H_1,H_2) = 2/3$.  
\end{case}

 In this case $\delta_\chi(H_1)=\delta_\chi(H_2)=0$. 
If $\chi(G_r[X_{\bm{v}}])\ge\sqrt{C}$, then we have the identical case as Case~5. Suppose then $\chi(G_b[X_{\bm{v}}])\ge\sqrt{C}$. 
Similarly by \cref{lmm:0-1/2}, $I_b^+=\varnothing$. Thus by the first part of \cref{cl:huoyibu} and \cref{claim:size}, we have 
$\big(\frac{2}{3}+\frac{\gamma}{4}\big)k<|I_r^{+}|\le \frac{2k}{3}$, a contradiction.

This completes the proof of~\cref{thm:main}.

\section{Concluding Remarks}\label{sec:rmk}

In this paper, we have completely determined the chromatic threshold for a pair of 3-chromatic graphs. A natural next step is to investigate a pair of general graphs. We believe that for any pair of general graphs, their chromatic threshold takes value $\frac{t}{t+2}$ for some $t\in\mathbb{N}$. In particular, we propose the following.

\begin{conj}
For any $r,s\ge 3$, there exists an interval of integers $[a,b]$ such that
\[
   \Big\{\delta_{\chi}(H_1,H_2):~ \chi(H_1)=r, ~\chi(H_2)=s\Big\}= \left\{\frac{t}{t+2}: t\in [a,b]\right\}.
\]
\end{conj}

\bibliographystyle{abbrv}
\bibliography{references.bib}

\end{document}